\documentclass{article}

\usepackage[utf8]{inputenc}
\pdfoutput=1
\usepackage[sc,osf]{mathpazo}
\usepackage[margin=1.1in,letterpaper]{geometry}
\usepackage{natbib}
\usepackage[colorlinks=true,breaklinks]{hyperref}
\usepackage[hyphenbreaks]{breakurl} 
\usepackage{color}

\definecolor{darkblue}{rgb}{0.1,0.1,0.75}
\hypersetup{
  colorlinks = true,
  linkcolor  = darkblue,
  citecolor  = darkblue,
  filecolor  = darkblue,
}

\makeatletter
\newlength\aftertitskip     \newlength\beforetitskip
\newlength\interauthorskip  \newlength\aftermaketitskip

\def\maketitle{\par
 \begingroup
   \def\thefootnote{\fnsymbol{footnote}}
   \def\@makefnmark{\hbox to 4pt{$^{\@thefnmark}$\hss}}
   \@maketitle \@thanks
 \endgroup
\setcounter{footnote}{0}
 \let\maketitle\relax \let\@maketitle\relax
 \gdef\@thanks{}\gdef\@author{}\gdef\@title{}\let\thanks\relax}

\def\@startauthor{\noindent \normalsize\bf}
\def\@endauthor{}
\def\@starteditor{\noindent \small {\bf Editor:~}}
\def\@endeditor{\normalsize}
\def\@maketitle{\vbox{\hsize\textwidth
 \linewidth\hsize \vskip \beforetitskip
 {\begin{center} \LARGE\@title \par \end{center}} \vskip \aftertitskip
 {\def\and{\unskip\enspace{\rm and}\enspace}%
  \def\addr{\small\it}%
  \def\email{\hfill\small\tt}%
  \def\name{\normalsize\bf}%
  \def\AND{\@endauthor\rm\hss \vskip \interauthorskip \@startauthor}
  \@startauthor \@author \@endauthor}
}}

\makeatother

\pdfoutput=1                    % force pdflatex to run

\usepackage{amsmath,amsthm}
\usepackage{graphicx}
\usepackage{color}
\usepackage{amssymb,amsfonts,amsxtra,mathrsfs,bm}
\usepackage{multicol}
\usepackage{multirow}
\usepackage{paralist}
\usepackage{xspace}
\usepackage{algorithm}% http://ctan.org/pkg/algorithm
\usepackage{algpseudocode}
\usepackage{hyperref}
\usepackage{bbm}
\usepackage{nicefrac}
\usepackage{subfig} % subfloat

\newcommand{\nfh}{\nicefrac{1}{2}}

\newcommand{\Mc}{\mathcal{M}}
\newcommand{\Xc}{\mathcal{X}}

\newcommand{\Bc}{\mathcal{B}}

\newcommand{\Tc}{\mathcal{T}}

\newcommand{\inv}[1]{#1^{-1}}

\newcommand{\nlsum}{\sum\nolimits}

\newcommand{\reals}{\mathbb{R}}
\newcommand{\rplus}{\reals_+}
\newcommand{\pd}{\mathbb{P}}
\newcommand{\riem}{\delta_R}
\newcommand{\gm}{{\#}}
\newcommand{\dist}{d}

\newcommand{\da}{\downarrow}
\newcommand{\ua}{\uparrow}

\newcommand{\ip}[2]{\langle {#1},\, {#2} \rangle}

\newcommand{\norm}[1]{\|{#1}\|}

\newcommand{\frob}[1]{\|{#1}\|_{\text{F}}}

\newcommand{\half}{\tfrac{1}{2}}

\newcommand{\set}[1]{\{ #1\}}

\newcommand{\rfw}{\textsc{Rfw}\xspace}
\newcommand{\efw}{\textsc{Efw}\xspace}
\newcommand{\fw}{\textsc{Fw}\xspace}

\DeclareMathOperator*{\argmin}{argmin}
\DeclareMathOperator*{\argmax}{argmax}

\DeclareMathOperator{\sgn}{sgn}
\DeclareMathOperator{\trace}{tr}

\DeclareMathOperator{\Exp}{Exp}
\DeclareMathOperator{\grad}{grad}

\newcommand{\matlab}{\textsc{Matlab}\xspace}

\renewcommand{\H}{\mathbb{H}}

\newtheorem{theorem}{Theorem}[section]
\newtheorem{lemma}[theorem]{Lemma}
\newtheorem{proposition}[theorem]{Proposition}

\theoremstyle{definition}
\newtheorem{defn}[theorem]{Definition}
\newtheorem{ass}{Assumption}
\newtheorem{rmk}[theorem]{Remark}

\numberwithin{equation}{section}

\title{Riemannian Optimization via Frank-Wolfe Methods}
\author{\name Melanie Weber \email{mw25@math.princeton.edu}\\
  \addr{Princeton University}\\
  \name Suvrit Sra \email{suvrit@mit.edu}\\
  \addr{Laboratory for Information and Decision Systems, MIT}
}

\usepackage{amsopn}

\begin{document}
\maketitle

% REQUIRED
\begin{abstract}
We study projection-free methods for constrained Riemannian optimization. In particular, we propose the Riemannian Frank-Wolfe (\rfw) method. We analyze non-asymptotic convergence rates of \rfw to an optimum for (geodesically) convex problems, and to a critical point for nonconvex objectives. We also present a practical setting under which \rfw can attain a linear convergence rate. 
As a concrete example, we specialize \rfw to the manifold of positive definite matrices and apply it to two tasks: (i) computing the matrix geometric mean (Riemannian centroid); and (ii) computing the Bures-Wasserstein barycenter. 
Both tasks involve geodesically convex interval constraints, for which we show that the Riemannian  ``linear'' oracle required by \rfw admits a closed form solution; this result may be of independent interest. We further specialize  \rfw to the special orthogonal group, and show that here too, the Riemannian ``linear'' oracle can be solved in closed form. Here, we describe an application to the synchronization of data matrices (\emph{Procrustes problem}). 
We complement our theoretical results with an empirical comparison of \rfw against state-of-the-art Riemannian optimization methods, and observe that \rfw performs competitively on the task of computing Riemannian centroids.
\end{abstract}

%%%%%%%%%%%%%%%%%%%%%%%%%%%%%%%%%
\section{Introduction}
We study the following constrained optimization problem
\begin{equation}
  \label{eq:0}
  \min_{x \in \Xc \subseteq \Mc}\quad \phi(x),
\end{equation}
where $\phi: \Mc \to \reals$ is a differentiable function and $\Xc$ is a compact g-convex subset of a Riemannian manifold $\Mc$. The objective $\phi$ may be geodesically convex (henceforth, \emph{g-convex}) or nonconvex. 
When the constraint set $\Xc$ is ``simple'' one may solve~\eqref{eq:0} via Riemannian projected-gradient. But in many cases, projection onto $\Xc$ can be expensive to compute, motivating us to seek \emph{projection-free} methods.

Euclidean ($\Mc \equiv \reals^n$) projection-free methods based on the Frank-Wolfe (\textsc{Fw}) scheme~\citep{fw_original} have recently witnessed a surge of interest in machine learning and related fields~\citep{jaggi2013revisiting,julien15}. Instead of projection, such \textsc{Fw} methods rely on access to a \emph{``linear'' oracle,} that is, a subroutine that solves the problem
\begin{equation}
  \label{eq:lo}
  \min_{z\in \Xc}\quad \ip{z}{\nabla\phi(x)},
\end{equation}
that can sometimes be much simpler than projection onto $\Xc$. The attractive property of \fw methods has been exploited in a variety of settings, including convex~\citep{jaggi2013revisiting,bach2015duality}, nonconvex~\citep{LJ16}, submodular~\citep{fujiso11,caliChVon11} and stochastic~\citep{hazan2016variance,reddi2016stochastic}; among numerous others. 

But as far as we are aware, \fw methods have not been studied for Riemannian manifolds. Our work fills this gap in the literature by developing, analyzing, and experimenting with Riemannian Frank-Wolfe (\rfw) methods. 
In addition to adapting \fw to the Riemannian setting, there is one more challenge that we must overcome: \rfw requires access to a Riemannian analog of the linear oracle~\eqref{eq:lo}, which can be hard even for g-convex problems.

Therefore, to complement our theoretical analysis of \rfw, we discuss in detail practical settings that admit efficient Riemannian ``linear'' oracles. Specifically, we discuss problems where $\Mc=\pd_d$, the manifold of (Hermitian) positive definite matrices, and $\Xc$ is a g-convex semidefinite interval; then problem~\eqref{eq:0} assumes the form
\begin{equation}
  \label{eq:8}
  \min_{X \in \Xc \subseteq \pd_d}\quad \phi(X),\qquad\text{where}\ \Xc := \set{X \in \pd_d \mid L \preceq X \preceq U},
\end{equation}
where $L$ and $U$ are positive definite matrices. An important instance of~\eqref{eq:8} is the following g-convex optimization problem (see \S\ref{sec:mgm} for details and notation):
\begin{equation}
  \label{eq:21}
  \min_{X \in \pd_d}\quad\sum_{i=1}^n w_i\riem^2(X,A_i),\quad\text{where}\ w \in \Delta_n, A_1,\ldots,A_n \in \pd_d,
\end{equation}
which computes the Riemannian centroid of a set of positive definite matrices (also known as the ``matrix geometric mean'' and the ``Karcher mean'')~\citep{bhatia07,karcher,LL14}. We will show that \rfw offers a simple approach to solve~\eqref{eq:21} that performs competitvely recently published state-of-the-art approaches. As a second application, we show that \rfw allows for an efficient computation of Bures-Wasserstein barycenters on the Gaussian density manifold. 

Furthermore, to offer an example on a matrix manifold other than positive definite matrices, we show that the Riemannian ``linear'' oracle can be solved in closed form for the special orthogonal group too. Specifically, we use this oracle to obtain an efficient \rfw based approach for the task of synchronizing data matrices, a classic problem better known as the \emph{Procrustes problem}.

\vskip12pt
\noindent\textbf{Summary of results}. The key contributions of this paper are as follows:
\begin{enumerate}
\item We introduce a Riemannian Frank-Wolfe (\rfw) algorithm for addressing constrained g-convex optimization on Riemannian manifolds.  We show that \rfw attains a \emph{non-asymptotic} $O(1/k)$ rate of convergence (Theorem~\ref{thm:conv-FWR}) to the optimum, $k$ being the number of iterations. Furthermore, under additional assumptions on the objective function and the constraints, we show that \rfw can even attain linear convergence rates (Theorem~\ref{thm:lin-conv-FWR}). In the nonconvex case, \rfw attains a \emph{non-asymptotic} $O(1/\sqrt{k})$ rate of convergence to first-order critical points (Theorem~\ref{thm:conv-FWR-nonconvex}). These rates are comparable to the best known guarantees for the classical Euclidean Frank-Wolfe algorithm~\cite{jaggi2013revisiting,julien15}.
\item While the Euclidean ``linear'' oracle is a convex problem, the Riemannian ``linear'' oracle is nonconvex.  Therefore, the key challenge of developing \rfw lies in efficiently solving the ``linear'' oracle. We address this problem with the following contributions:
\begin{itemize}
\item We specialize \rfw for g-convex problems of the form~\eqref{eq:8} on the manifold of Hermitian positive definite (HPD) matrices. Most importantly, for this problem we develop a closed-form solution to the Riemannian ``linear'' oracle,  which involves solving a nonconvex semi-definite program (\emph{SDP}), see Theorem~\ref{thm:logtrace}.  We then apply \rfw to computing the Riemannian mean of HPD matrices. In comparison with state-of-the-art methods, we observe that \rfw   performs competitively. Furthermore, we implement \rfw for the computation of Bures-Wasserstein barycenters on the Gaussian density manifold. 
\item We specialize \rfw for optimization on the manifold of special orthogonal matrices, where the Riemannian ``linear'' oracle can be solved in closed form too. Notably, this problem is an instance of \rfw being applicable to a setting with nonconvex constraints. We present an application to the classical Procrustes problem.
\item We show that we can recover a sublinear convergence rate, even if the Riemannian ``linear'' oracle can only be solved approximately, e.g., using relaxations or iterative solvers. This makes the approach applicable to a wide range of constrained optimization problems.
\end{itemize}
\end{enumerate}
We believe the closed-form solutions for the Riemannian ``linear'' oracle, which involve a nonconvex SDP, respectively, should be of wider interest too. A similar approach can be used to solve the Euclidean linear oracle, a convex SDP, in closed form (Appendix~\ref{sec:e-lo}). These oracles lie at the heart of the competitive performance of \rfw and \fw in comparison with state-of-the-art methods for computing the Riemannian centroid of HPD matrices (Section~\ref{sec:expts}). More broadly, we hope that our results encourage others to study \rfw as well as other examples of problems with efficient Riemannian ``linear'' oracles.\\

\vskip5pt
\noindent\textbf{Related work}. Riemannian optimization has a venerable history. The books~\citep{udriste1994convex,absil2009optimization} provide a historical perspective as well as basic theory.  The focus of these books and of numerous older works on Riemannian optimization, e.g.,~\citep{edelman98,helmke2007essential,ring2012optimization}, is almost exclusively on asymptotic analysis. More recently, non-asymptotic  convergence analysis quantifying iteration complexity of Riemannian optimization algorithms has begun to be pursued~\citep{ZS16,boumal2016global,bento2017iteration}. However, to our knowledge, all these works either focus on \emph{unconstrained} Riemannian optimization, or handle constraints via projections. In contrast, we explore constrained g-convex optimization within an abstract \rfw framework, by assuming access to a Riemannian ``linear'' oracle.  Several applications of Riemannian optimization are known, including to matrix factorization on fixed-rank manifolds~\citep{vandereycken2013low,tan2014riemannian}, dictionary learning \citep{cherian2015riemannian,sun2015complete}, classical optimization under orthogonality constraints \citep{edelman98}, averages of rotation matrices~\citep{moakher2002means}, elliptical distributions in statistics~ \citep{zhang2013multivariate,sra2013geometric}, and Gaussian mixture models \cite{hoSr15b}. Explicit theory of g-convexity on HPD matrices is studied in~\citep{SH15}. Additional related work corresponding to the Riemannian mean of HPD matrices is discussed in Section~\ref{sec:mgm}.

%%%%%%%%%%%%%%%%%%%%%%%%%%%%%%%%%

%\vspace{-12pt}
\section{Background}
We begin by noting some useful background and notation from Riemannian geometry. For a deeper treatment we refer the reader to~\citep{jost,chavel2006riemannian}.

A smooth \emph{manifold} $\mathcal{M}$ is a locally Euclidean space equipped with a differential structure. At any point $x \in \mathcal{M}$, the set of tangent vectors forms the \emph{tangent space} $T_x\Mc$. Our focus is on \emph{Riemannian manifolds}, i.e., smooth manifolds with a smoothly varying inner product $\ip{\xi}{\eta}_x$ defined on the $T_x\Mc$ at each point $x\in \Mc$. We write $\|\xi\|_x := \sqrt{\ip{\xi}{\xi}_x}$ for $\xi\in T_x\Mc$; for brevity, we will drop the subscript on the norm whenever the associated tangent space is clear from context. Furthermore, we assume $\Mc$ to be \emph{complete},  which ensures that the following map is defined on the whole tangent space:
We define the \emph{exponential map} as a mapping from $T_x\Mc$ to $\Mc$; it is defined by $\Exp_x: T_x \mathcal{M} \rightarrow \mathcal{M}$ such that $y=\Exp_x(g_x) \in \mathcal{M}$
along a geodesic $\gamma : [0,1] \rightarrow \mathcal{M}$ with $\gamma(0)= x$, $\gamma(1)=y$ and $\dot{\gamma}(0)=g_x \in T_x\Mc$. We can define an \emph{inverse} exponential map $\Exp_x^{-1}: \mathcal{M} \rightarrow T_x \mathcal{M}$ as a diffeomorphism from the neighborhood of $x \in \mathcal{M}$ onto the neighborhood of $0 \in T_x \mathcal{M}$ with $\Exp_x^{-1}(x)=0$. Note, that the completeness of $\mathcal{M}$ ensures that both maps are well-defined.

Since tangent spaces are local notions, one cannot directly compare vectors lying in different tangent spaces. To tackle this issue, we use the concept of \emph{parallel transport}: the idea is to transport (map) a tangent vector along a geodesic to the respective other tangent space. More precisely, let $x, y \in \mathcal{M}$ with $x \neq y$. We transport $g_x \in T_x \mathcal{M}$ along a geodesic $\gamma$ (where $\gamma(0)=x$ and $\gamma(1)=y$) to the tangent space $T_y\Mc$; we denote this by $\Gamma_x^y g_x$. Importantly, the inner product on the tangent spaces is preserved under parallel transport, so that $\ip{\xi_x}{\eta_x}_x = \ip{\Gamma_x^y \xi_x}{\Gamma_x^y \eta_x}_y$, 
where $\xi_x, \eta_x \in T_x \mathcal{M}$, while $\ip{\cdot}{\cdot}_x$ and  $\ip{\cdot}{\cdot}_y$ are the respective inner products.

%\vspace{-12pt}
\subsection{Gradients, convexity, smoothness}
Recall that the \emph{Riemannian gradient} $\grad\phi(x)$ is the unique vector in $T_x\Mc$ such that the directional derivative
\begin{equation*}
  D\phi(x)[v] = \ip{\grad\phi(x)}{v}_x,\qquad\forall v \in T_x\Mc.
\end{equation*}
When optimizing functions using gradients, it is useful to impose some added structure. The two main properties we require are sufficiently smooth gradients and geodesic convexity. We say $\phi: \mathcal{M} \rightarrow \mathbb{R}$ is \emph{$L$-smooth}, or that it has \emph{$L$-Lipschitz} gradients, if
\begin{equation}
  \label{eq:17}
  \|\grad\phi(y)-\Gamma_x^y\grad\phi(x)\| \le L d(x,y) \qquad \forall\ x, y \in \Mc; \; g_x \in T_x\Mc \; ,
\end{equation}
where $d(x,y)$ is the geodesic distance between $x$ and $y$; equivalently, 
\begin{equation}
\label{eq:L}
  \phi(y) \leq \phi(x) + \ip{g_x}{\Exp_x^{-1} (y)}_x + \tfrac{L}{2} d^2(x,y) \qquad \forall x,y \in \mathcal{M}; \; g_x \in T_x\Mc \; .
\end{equation}
We say $\phi: \mathcal{M} \rightarrow \mathbb{R}$ is \emph{geodesically convex} (\emph{g-convex}) if
\begin{equation}
\phi(y) \geq \phi (x) + \ip{g_x}{\Exp_x ^{-1}(y)}_x \qquad \forall x,y \in \mathcal{M} ; \; g_x \in T_x\Mc \; ,
\end{equation}
and call it \emph{$\mu$-strongly g-convex} ($\mu \ge 0$) if
\begin{equation}
\phi(y) \geq \phi(x) + \ip{g_x}{\Exp_x ^{-1}(y)}_x + \tfrac{\mu}{2} d^2(x,y) \qquad \forall x,y \in \mathcal{M} ; \; g_x \in T_x\Mc \; .
\end{equation}

The following observation underscores the reason why g-convexity is a valuable geometric property for optimization.
\begin{proposition}[Optimality]
  \label{prop.opt}
  Let $x^* \in\Xc\subset \Mc$ be a local optimum for~\eqref{eq:0}. Then, $x^*$ is globally optimal, and $\ip{\grad\phi(x^*)}{\Exp_{x^*}^{-1}(y)}_{x^*} \ge 0$ for all $y\in \Xc$.
\end{proposition}

\subsection{Projection-free vs.\ Projection-based methods}
The growing body of literature on Riemannian optimization considers mostly projection-based methods, such as \emph{Riemannian Gradient Decent} (\textsc{RGD}) or \emph{Riemannian Steepest Decent} (\textsc{RSD})~\citep{absil2009optimization}. 
Such methods and their convergence guarantees typically require Lipschitz assumptions. However, the objectives of many classic optimization and machine learning tasks are not Lipschitz on the whole manifold. In such cases, an additional compactness argument is required.  However, in projection-based methods, the typically used retraction back onto the manifold may not be guaranteed to land in this compact set.  Thus, additional and potentially expensive work (e.g., an additional projection step) is needed to ensure that the update remains in the compact region where the gradient is Lipschitz. On the other hand, projection-free methods, such as \fw, bypass this issue, because their update is guaranteed to stay within the compact feasible region. Importantly, in some problems, the Riemannian linear oracle at the heart of \fw can be much less expensive than computing a projection back onto the compact set.  A detailed numerical study, comparing the complexity of projections with that of computing linear minimizers in the Euclidean case, can be found in~\cite{COMBETTES2021565}.
The efficiency linear minimizers is especially significant for the applications highlighted in this paper, where the linear oracle even admits a closed form solution (see sections~\ref{sec:r-lo} and~\ref{sec:SO(n)-LO}).

\subsection{Constrained Optimization in Riemannian Space}
\label{sec:constrained-examples}
Many applications in machine learning, statistics and data science involve constrained optimization tasks, which can benefit from a Riemannian perspective.  A large body of literature has considered the problem of translating a constrained Euclidean optimization problem into an \emph{unconstrained} Riemannian problem, by encoding the primary constraint in the manifold structure. However, oftentimes, a problem has additional constraints,  requiring a Riemannian approach to constrained optimization.  We list below notable examples, including those that will be covered in the application section of the paper.

Examples on the manifold of positive definite matrices include the computation of Riemannian centroid (with interval constraints, see section~\ref{sec:mgm.real}) and learning determinantal point processes (with interval constraints~\cite{Mariet2015FixedpointAF}).  A related problem is that of computing Wasserstein-Barycenters on the Bures manifold (wirh interval constraints, see section~\ref{sec:Bures}).  The $k$-means clustering algorithm corresponds to an optimization task on the Stiefel manifold with equality and inequality constraints~\cite{villar}.  Non-negative PCA can be computed on the sphere with equality constraints~\cite{montanari}.  The synchronization of data matrices can be written as an optimization task on the manifold of the orthogonal group with a determinant constraint (see section~\ref{sec:SO(d)}).  Computing a minimum balanced cut for graph bisection can be computed on the Oblique manifold with quadratic equality constraints~\cite{liu2019simple}.

%%%%%%%%%%%%%%%%%%%%%%%%%%%%
\section{Riemannian Frank-Wolfe}
The condition $\ip{\grad\phi(x^*)}{\Exp_{x^*}^{-1}(y)}_{x^*} \ge 0$ for all $y\in \Xc$ in Proposition~\ref{prop.opt} lies at the heart of Frank-Wolfe (also known as ``conditional gradient'') methods. In particular, if this condition is not satisfied, then there must be a \emph{feasible descent direction} --- \fw schemes seek such a direction and update their iterates~\citep{fw_original,jaggi2013revisiting}. This high-level idea is equally valid in the Riemannian setting. Algorithm~\ref{alg.fw} recalls the basic (Euclidean) \fw method which solves $\min_{x\in \Xc}\phi(x)$, and Algorithm~\ref{alg.rfw} introduces its Riemannian version \rfw, obtained by simply replacing Euclidean objects with their Riemannian counterparts.
In the following, $\ip{\cdot}{\cdot}$ will denote $\ip{\cdot}{\cdot}_{x_k}$ unless otherwise specified.
\begin{algorithm}[H]
  \caption{Euclidean Frank-Wolfe without line-search}
  \label{alg.fw}
  \begin{algorithmic}[1]
    \State Initialize with a feasible point $x_0 \in \Xc \subset \reals^n$
    \For{$k=0,1,\ldots$}
      \State Compute $z_k \gets \argmin_{z \in \Xc} \ip{\nabla\phi(x_k)}{z-x_k}$
      \State Let $s_k \gets \frac{2}{k+2}$
      \State Update $x_{k+1} \gets (1-s_k)x_k + s_kz_k$
    \EndFor
  \end{algorithmic}
\end{algorithm}
Notice that to implement Algorithm~\ref{alg.fw}, $\Xc$ must be compact. Convexity ensures that after the update in Step 5, $x_{k+1}$ remains feasible, while compactness ensures that the \emph{linear oracle} in Step 3 has a solution. To obtain \rfw, we first replace the linear oracle (Step 3 in Algorithm~\ref{alg.fw}) with the \emph{Riemannian ``linear oracle''}:
\begin{equation}
  \label{eq:16}
  \min_{z \in \Xc}\quad \ip{\grad\phi(x_k)}{\Exp_{x_k}^{-1}(z)},
\end{equation}
where now $\Xc$ is assumed to be a compact g-convex set. Similarly, observe that Step 5 of Algorithm~\ref{alg.fw} updates the current iterate $x_k$ along a straight line joining $x_k$ with $z_k$. Thus, by analogy, we replace this step by moving $x_k$ along a geodesic joining $x_k$ with $z_k$. The resulting \rfw algorithm is presented as Alg~\ref{alg.rfw}.
\begin{algorithm}[htbp]
  \caption{Riemannian Frank-Wolfe (\rfw) for g-convex optimization}
  \label{alg.rfw}
  \begin{algorithmic}[1] 
     \State Initialize $x_0 \in \Xc \subseteq \mathcal{M}$; assume access to the geodesic map $\gamma: [0,1]\to \Mc$
     \For {$k=0,1,\dots$}
        \State $z_k \gets \argmin_{z \in \Xc}\ \ip{\grad\phi(x_k)}{\Exp_{x_k}^{-1}(z)}$
        \State Let $s_k \gets \frac{2}{k+2}$
        \State $x_{k+1} \gets \gamma(s_k)$, where $\gamma(0)=x_k$ and $\gamma(1)=z_k$
     \EndFor
   \end{algorithmic}
\end{algorithm}
While we obtained Algorithm~\ref{alg.rfw} purely by analogy, we must still show that this analogy results in a valid algorithm. In particular, we need to show that Algorithm~\ref{alg.rfw} converges to a solution of~\eqref{eq:0}. We will in fact prove a stronger result that \rfw converges globally at the rate $O(1/k)$, i.e., $\phi(x_k)-\phi(x^*) = O(1/k)$, which matches the rate of the usual Euclidean \fw method.

\subsection{Convergence analysis}
We make the following smoothness assumption:
\begin{ass}[Smoothness]
  \label{A.smooth}
  The objective $\phi$ has a locally Lipschitz continuous gradient on $\Xc$, that is, there exists a constant $L$ such that for all $x, y \in \Xc$ we have
  \begin{equation}
  \label{eq:3}
  \norm{\grad\phi(x)-\Gamma_y^x\grad\phi(y)} \le L \; \dist{(x,y)}.
\end{equation}
\end{ass}
Next, we introduce a quantity that will play a central role in the convergence rate of \rfw, namely the \emph{curvature constant} 
\begin{equation}
  \label{eq:defm}
  M_\phi := \sup_{\substack{x,y,z \in \Xc}} 
  \tfrac{2}{\eta^2}\left[\phi(y) - \phi(x) - \ip{\grad \phi(x)}{\Exp_x ^{-1}(y)} \right] \; .
\end{equation}
An analogous quantity is used in the analysis of the Euclidean \fw~\cite{jaggi2013revisiting}. In the following we adapt proof techniques from~\cite{jaggi2013revisiting} to the Riemannian setting.
Here and in the following, $y=\gamma(\eta)$ for some $\eta \in [0,1]$ and a geodesic map $\gamma:[0,1] \rightarrow \Mc$ with $\gamma(0)=x$ and $\gamma(1)=z$ (denoted in the following as $\gamma_{xy}$).
Lemma~\ref{lem:bound-on-M} relates the curvature constant~\eqref{eq:defm} to the Lipschitz constant $L$.
\begin{lemma}
\label{lem:bound-on-M}
Let $\phi: \Mc \to \reals$ be $L$-smooth on $\Xc$, and let $\mathrm{diam}(\Xc) := \displaystyle\sup_{x,y \in \Xc}\dist{(x,y)}$. Then, 
the curvature constant $M_\phi$ satisfies the bound
\begin{equation*}
  M_{\phi} \leq L \; {\rm diam}(\Xc)^2.
\end{equation*}
\end{lemma}
\begin{proof}
Let $x, z \in \Xc$ and $\eta \in (0,1)$; let $y=\gamma_{xz}(\eta)$ be a point on the geodesic joining $x$ with $z$. This implies $\frac{1}{\eta^2} \dist{(x,y)}^2 = \dist{(x,z)}^2$. From~\eqref{eq:L} we know that $\norm{\phi(z) - \phi(x) - \ip{\grad \phi(x)}{\Exp_x ^{-1} (z)}}^2 \leq \frac{L}{2} \dist{(x,z)}^2$, 
whereupon using the definition of the curvature constant we obtain
\begin{equation}
M_{\phi} \leq \sup \frac{2}{\eta^2} \frac{L}{2} \dist{(x,y)}^2 = \sup L \; \dist{(x,z)}^2 \leq L \cdot {\rm diam} (\Xc)^2 \; .
\end{equation}
\end{proof}

We note below an analog of the Lipschitz inequality~\eqref{eq:L} using the constant $M_\phi$.
\begin{lemma}[Lipschitz]
  \label{lem:lip-bound-gen}
  Let $x, y, z \in \Xc$ and $\eta \in [0,1]$ with $y=\gamma_{xz}(\eta)$. Then,
  \begin{align*}
    \phi(y) \leq \phi(x) + \eta \ip{\grad \; \phi(x)}{\Exp_x^{-1}(z)} + \tfrac{1}{2} \eta^2 M_{\phi} \; .
  \end{align*}
\end{lemma}
\begin{proof}
  From definition~\eqref{eq:defm} of the constant $M_\phi$ we see that
  \begin{align*}
    M_{\phi} \geq \tfrac{2}{\eta^2} \left( \phi(y) - \phi(x) - \ip{\grad \; \phi(x)}{\Exp_{x}^{-1}(y)}	\right),
  \end{align*}
  which we can rewrite as
  \begin{equation}
    \label{eq:23}
\phi(y) \leq \phi(x) + \ip{\grad \; \phi(x)}{\Exp_{x}^{-1}(y)} + \tfrac{1}{2} \eta^2 M_{\phi}.
\end{equation}
Furthermore, since $y=\gamma_{xz}(\eta)$, we have $\Exp_x^{-1} (y) = \eta \Exp_x^{-1} (z)$, and therefore
  \begin{align*}
    \ip{\grad \phi(x)}{\Exp_x^{-1}(y)} &= \ip{\grad \phi(x)}{\eta \Exp_x^{-1}(z)} = \eta \ip{\grad \phi(x)}{\Exp_x^{-1}(z)}.
  \end{align*}
  Plugging this equation into~\eqref{eq:23} the claim follows.
\end{proof}

We need one more technical lemma (easily verified by a quick induction).
\begin{lemma}[Stepsize for \textsc{RFW}]
  \label{lem:induction}
  Let $(a_k)_{k \in I}$ be a nonnegative sequence that fulfills
  \begin{equation}
    \label{eq:4}
    a_{k+1} \le (1-s_k)a_k + \half s_k^2 M_{\phi}.
  \end{equation}
  If $s_k = \frac{2}{(k+2)}$, then, $a_k \le \frac{2M_{\phi}}{(k+2)}$.
\end{lemma}

\iffalse %{
\begin{proof}
   For $k=0$, we have $\gamma_0=1$, so~\eqref{eq:4} yields $a_1 \le \half M_{\phi}$ which is clearly smaller than $\frac{2}{3}M_{\phi}$, so the base case holds. Next, assume the induction hypothesis, whereby
  \begin{equation*}
    \begin{split}
      a_{k+1} &\le (1-s_k)a_k + \half s_k^2M_{\phi}\\
      &\le \frac{k}{k+2}\frac{2M_{\phi}}{k+2} + \frac{2}{(k+2)^2}M_{\phi}\\
      & = \frac{2M_{\phi}}{k+3}\left(\frac{k(k+3)}{(k+2)^2} + \frac{k+3}{(k+2)^2}\right)\\
      & = \frac{2M_{\phi}}{k+3}\left(\frac{k^2+4k+3}{k^2+4k+4}\right) \le \frac{2M_{\phi}}{k+3},
    \end{split}
  \end{equation*}
  which completes the inductive step and hence the proof. %\qed
\end{proof}
% }
\fi

We are now ready to state our first main convergence result, Theorem \ref{thm:conv-FWR} that establishes a \emph{global} iteration complexity for \rfw.
\begin{theorem}[Rate]
  \label{thm:conv-FWR}
  Let $s_k = \frac{2}{k+2}$, and let $X^*$ be a minimum of $\phi$. Then, the sequence of iterates $X_k$ generated by Algorithm~\ref{alg.rfw} satisfies  $\phi(X_k)-\phi(X^*) = O(1/k)$.
\end{theorem}
\begin{proof}
  The proof of this claim is relatively straightforward; indeed
  \begin{align*}
    &\phi(X_{k+1}) - \phi(X^*)\\
    &\le \phi(X_k) - \phi(X^*) + s_k\ip{\grad \phi(X_k)}{\Exp_{X_k} ^{-1}(Z_k)} + \half s_k^2M_\phi\\
    &\le \phi(X_k) - \phi(X^*) + s_k\ip{\grad \phi(X_k)}{\Exp_{X_k} ^{-1}(X^*)} + \half s_k^2M_\phi\\
    &\le \phi(X_k)-\phi(X^*) - s_k(\phi(X_k)-\phi(X^*)) + \half s_k^2M_\phi\\
    &=(1-s_k)(\phi(X_k)-\phi(X^*)) + \half s_k^2M_\phi,
  \end{align*}
  where the first inequality follows from Lemma~\ref{lem:lip-bound-gen}, while the second one from $Z_k$ being an $\argmin$ obtained in Step~3 of the algorithm. The third inequality follows from g-convexity of $\phi$. 
  Setting $a_k=\phi(X_k)-\phi(X^*)$ in Lemma~\ref{lem:induction} we immediately obtain 
  $$
  \phi(X_k)-\phi(X^*) \le \frac{2M_\phi}{k+2},\qquad k\ge 0,
  $$
  which is the desired $O(1/k)$ convergence rate.
\end{proof}

Theorem~\ref{thm:conv-FWR} provides a global sublinear convergence rate for \rfw. Typically, \fw methods trade off their simplicity for this slower convergence rate, and even for smooth strongly convex objectives they do not attain linear convergence rates~\citep{jaggi2013revisiting}. We study in Section~\ref{sec:linear} a setting that permits \rfw to attain a linear rate of convergence.

\subsection{Linear convergence of RFW}
\label{sec:linear}
In general, the sublinear convergence rate that we derived in the previous section is best-possible for Frank-Wolfe methods. This is due to the following phenomenon, which has been studied extensively in the Euclidean setting~\citep{canon-cullum,wolfe}: If the optimum lies on the boundary of the constraint set $\Xc$, then the \fw updates will ``zig-zag'',  resulting in a slower convergence. In this case, the upper bound on the global convergence rate is tight. If, however, the optimum lies in the strict interior of the constraint set, Euclidean \fw is known to converge at a linear rate \citep{marcotte,GH15}.  
Remarkably, under a similar assumption, \rfw also displays global linear convergence, which we will formally prove below (Theorem~\ref{thm:lin-conv-FWR}). Notably,  for the special case of the \emph{geometric matrix mean} that we analyze in the next section, this strict interiority assumption will always be valid, provided that not all the matrices are the same.  

%\subsubsection{Linear convergence under strict interior assumption}
We will make use of a Riemannian extension to the well-known Polyak-Łojasiewicz (PL) inequality~\citep{polyak63,losj63}, which we define below. 
Consider the minimization $$\min_{x\in \Mc} f(x),$$ and let $f^{*}$ be the optimal function value. We say that $f$ satisfies the PL inequality if for some $\mu > 0$,
\begin{equation}
  \tfrac{1}{2} \norm{\grad f(x)}^2 \geq \mu \left(f(x)-f^{*}\right) \quad \forall x,y \in \Mc.
  \label{eq:PL}
\end{equation}
Inequality~\eqref{eq:PL} is weaker than strong convexity (and is in fact implied by it). It has been 
widely used for establishing linear convergence rates of gradient-based methods; see  \citep{KNS16} for several (Euclidean) examples, and \citep{rsvrg} for a Riemannian example.
We will make use of inequality~\eqref{eq:PL} for obtaining linear convergence of \rfw, by combining it with a strict interiority condition on the minimum.
\begin{theorem}[Linear convergence RFW]
  \label{thm:lin-conv-FWR}
  Suppose that $\phi$ is strongly g-convex with constant $\mu$ and that its minimum lies in a ball of radius $r$ that strictly inside the constraint set $\Xc$. Define $\Delta_k := \phi(X_k)-\phi(X^*)$ and let the step-size $s_k = \frac{r\sqrt{\mu \Delta_k}}{\sqrt{2}M_\phi}$. Then, \rfw converges linearly since it satisfies the bound 
  \begin{equation*}
    \Delta_{k+1} \le \left(1 - \frac{r^2\mu}{4M_\phi}\right)\Delta_k.
    \label{eq:linear-FWR}
  \end{equation*}
\end{theorem} 
\begin{proof}
  Let $\Bc_r(X^*) \subset \Xc$ be a ball of radius $r$ containing the optimum. Let 
\begin{equation*}
W_k := \argmax_{\xi \in \Tc_{X_k}, \norm{\xi}\le 1}\ip{\xi}{\grad \phi(X_k)}
\end{equation*}
be the direction of steepest descent in the tangent space $\Tc_{X_k}$. The point $P_k = \Exp_{X_k}(rW_k)$  lies in $\Xc$. Consider now the following inequality
\begin{equation}\label{eq:pt}
  \begin{split}
    \ip{-\Exp_{X_k}^{-1}(P_k)}{\grad \phi(X_k)} 
    &=
   -\ip{\grad \phi(X_k)}{r W_k} = - r\norm{\grad \phi(X_k)},
  \end{split}
\end{equation}
which follows upon using the definition of $W_k$. Thus, we have the bound
\begin{equation*}
  \begin{split}
    \Delta_{k+1} &\le \Delta_k + s_k \ip{\grad \phi(X_k)}{\Exp_{X_k}^{-1}(X_{k+1})} + \half s_k^2M_\phi\\
    &\le \Delta_k - s_k r\norm{\grad \phi(X_k)} +   \half s_k^2M_\phi\\
    &\le \Delta_k - s_k r\sqrt{2\mu}\sqrt{\Delta_k} + \half s_k^2M_\phi,
  \end{split}
\end{equation*}
where the first inequality follows from the Lipschitz-bound (Lemma~\ref{lem:lip-bound-gen}), the second one from \eqref{eq:pt}, and the third one from the PL inequality (which, in turn holds due to the $\mu$-strong g-convexity of $\phi$).
Now setting the step size $s_k = \frac{r\sqrt{\mu \Delta_k}}{\sqrt{2}M_\phi}$, we obtain
\begin{equation*}
  \Delta_{k+1} \le \left(1 - \frac{r^2\mu}{4M_\phi}\right)\Delta_k,
\end{equation*}
which delivers the claimed linear convergence rate. 
\end{proof}

Theorem~\ref{thm:lin-conv-FWR} provides a setting where \rfw can converge fast, however, it uses step sizes $s_k$ that require knowing $\phi(X^*)$\footnote{This step size choice is reminiscent of the so-called ``Polyak stepsizes'' used in the convergence analysis of subgradient methods~\citep{polyak}.}; in case this value is not available, we can use a worse value, which will still yield the desired inequality. 

\begin{rmk}[Necessity of strict interior assumption]
\normalfont
For optimization 
tasks with polytope constraints that do not fulfill a strict interior assumption as the one described above,  several \efw variants achieve linear convergence~\cite{julien15}.  Notable examples include 
\emph{Away-step \fw}~\cite{wolfe,marcotte},
\emph{Pairwise \fw}~\cite{mitchell74} and \emph{Fully-corrective \fw}~\cite{holloway}. One may ask, whether these variants can be generalized to the Riemannian case.  The first difficulty lies in finding a Riemannian equivalent of the polytope constraint set, which is defined as the convex hull of a finite set of vectors (\emph{atoms}).  Naturally, we could consider the convex hull of a finite set of points on a manifold,  which is the intersection of all convex set that contain them.  Unfortunately,  such a set is in general not compact -- compactness is only guaranteed for Hadamard manifolds under additional, restrictive conditions~\cite{ledyaev2006helly}. Even in this special case,  ensuring that the resulting constraint sets have sufficiently ``good'' geometry is not straightforward.
\end{rmk}

%%%%%%%%%%%
\subsection{RFW for nonconvex problems}
Finally, we want to consider the case where $\phi$ in Eq.~\ref{eq:0} may be convex. We cannot hope to find the global minimum with first-order methods, such as \rfw. However, we can compute first-order critical point via \rfw. For the setup and analysis, we follow closely S. Lacoste-Julien's analysis of the Euclidean case~\citep{LJ16}. 

We first introduce the \emph{Frank-Wolfe gap} as a criterion for evaluating convergence rates. For $X \in \Xc$, we write
\begin{align}
G(X) := \max_{Z \in \Xc} \ip{\Exp_Z^{-1}(X)}{- \grad \; \phi(X)} \; .
\end{align}
With this, we can show the following sublinear convergence guarantee:
\begin{theorem}[Rate (nonconvex case)]
  \label{thm:conv-FWR-nonconvex}
  Let $\tilde{G}_k := \min_{0 \leq k \leq K} G(X_k)$ (where $G(X_k)$ denotes the Frank-Wolfe gap at $X_k$). After $K$ iterations of Algorithm~\ref{alg.rfw}, we have $\tilde{G}_k \leq \frac{\max \lbrace	2 h_0, M_{\phi}	\rbrace}{\sqrt{K+1}}$.
\end{theorem}
We defer the proof to appendix~\ref{app:A}.

%%%%%%%%%%%%%%%%%%%%%%%%%%%%%%%%
\section{Specializing RFW for HPD matrices}
\label{sec:mgm}
In this section we study a concrete setting for \rfw, namely, a class of g-convex optimization problems with Hermitian positive definite (HPD) matrices. The most notable aspect of the concrete class of problems that we study is that the Riemannian linear oracle~\eqref{eq:16} will be shown to admit an efficient solution, thereby allowing an efficient implementation of Algorithm~\ref{alg.rfw}. The concrete class of problems that we consider is the following:
\begin{equation}
\label{eq:18}
  \min_{X \in \Xc \subseteq \pd_d}\quad \phi(X),\qquad\text{where}\ \Xc := \set{X \in \pd_d \mid L \preceq X \preceq U},
\end{equation}
where $\phi$ is a g-convex function and $\Xc$ is a ``positive definite interval'' (which is easily seen to be a g-convex set). Note that set $\Xc$ actually does not admit an easy projection for matrices (contrary to the scalar case). Problem~\eqref{eq:18} captures several g-convex optimization problems, of which perhaps the best known is the task of computing the matrix geometric mean (also known as the the \emph{Riemannian centroid} or \emph{Karcher mean})---see Section~\ref{sec:mgm.real}.

We briefly recall some facts about the Riemannian geometry of HPD matrices below. For a comprehensive overview, see, e.g., ~\cite[Chapter 6]{bhatia07}. We denote by $\H_d$ the set of $d\times d$ Hermitian matrices. The most common (nonlinear) Riemannian geometry on $\pd_d$ is induced by
\begin{equation}
  \label{eq:13}
  \ip{A}{B}_X := \trace(\inv{X}A\inv{X}B),\quad\ \text{where}\ A, B \in \H_d.
\end{equation}
This metric induces the geodesic $\gamma: [0,1]\to \pd_d$ between $X, Y \in \pd_d$ given by
\begin{equation}
  \label{eq:14}
  \gamma(t) := X^{\nfh}(X^{-\nfh}YX^{-\nfh})^{t}X^{\nfh},\qquad t \in [0,1].
\end{equation}
The corresponding \emph{Riemannian distance} is
\begin{equation}
  \label{eq:15}
  d(X,Y) := \frob{\log(X^{-\nfh}YX^{-\nfh})},\qquad X, Y \in \pd_d.
\end{equation}
The Riemannian gradient $\grad\phi$ is obtained from the Euclidean one ($\nabla\phi$) as follows
\begin{equation}
  \label{eq:10}
  \grad\phi(X) = X\nabla^\H\phi(X)X \; ,
\end{equation}
where $\nabla^\H\phi(X) := \nabla\phi(X) + (\nabla\phi(X))^*$
denotes the (Hermitian) symmetrization of the gradient.  The exponential map and its inverse at a point $P \in \pd_d$ are respectively given by
\begin{align*}
  \Exp_X(A) &= X^{\nfh}\exp(X^{-\nfh}A X^{-\nfh})X^{\nfh},\qquad A \in T_X\pd_d \equiv \H_d,\\
  \Exp_X^{-1}(Y) &= X^{\nfh}\log(X^{-\nfh}YX^{-\nfh})X^{\nfh},\qquad X, Y \in \pd_d,
\end{align*}
where $\exp(\cdot)$ and $\log(\cdot)$ denote the matrix exponential and logarithm, respectively. Observe that using \eqref{eq:10} we obtain the identity
\begin{equation}
  \label{eq:9}
  \ip{\grad\phi(X)}{\Exp_X^{-1}(Y)}_X = \ip{X^{\nfh}\nabla^\H\phi(X)X^{\nfh}}{\log(X^{-\nfh}YX^{-\nfh})}.
\end{equation}
With these details, Algorithm~\ref{alg.rfw} can almost be applied to~\eqref{eq:18} --- the most crucial remaining component is the Riemannian linear oracle, which we now describe.

\subsection{Solving the Riemannian linear oracle}
\label{sec:r-lo}
For solving \eqref{eq:18}, the Riemannian linear oracle (see (\ref{eq:16})) requires solving
\begin{equation}
  \label{eq:19}
  \min_{L \preceq Z \preceq U}\quad \ip{X_k^{\nfh}\nabla^\H\phi(X_k) X_k^{\nfh}}{\log(X_k^{-\nfh}ZX_k^{-\nfh})}.
\end{equation}
Problem~\eqref{eq:19} is a non-convex optimization problem over HPD matrices. However, remarkably, it turns out to have a closed form solution. Theorem~\ref{thm:logtrace} presents this solution and is our main technical result for Section~\ref{sec:mgm}. % Equivalently, this can be written 
\begin{theorem}
  \label{thm:logtrace}
  Let $L, U \in \pd_d$ such that $L \prec U$. Let $S \in \H_d$ and $X \in \pd_d$ be arbitrary. Then, the solution to the optimization problem
  \begin{equation}
    \label{eq.12.R}
    \min_{L \preceq Z \preceq U}\quad \trace(S \log(XZX)),
  \end{equation}
  is given by $Z = X^{-1}Q \left( P^* [-\sgn(D)]_+ P  + \hat{L}\right) Q^*X^{-1}$, where $S=QDQ^*$ is a diagonalization of $S$, $\hat{U} - \hat{L}=P^* P$ with $\hat{L}=Q^* X L X Q$ and $\hat{U}=Q^* X U X Q$.
\end{theorem}

For the proof of Theorem~\ref{thm:logtrace}, we need a fundamental lemma about eigenvalues of Hermitian matrices (Lemmas~\ref{lem.logmaj}). First,  we need to introduce some additional notation. For $x \in \mathbb{R}^d$, let $x \da = \left( x_1 \da, \dots, x_d \da \right)$ denote the vector with entries of $x$ in decreasing order, i.e.,  $x_1 \da \geq ... \geq x_d \da$.  For $x, y \in \mathbb{R}^d$ we say that $x$ is majorized by $y$ ($x \prec y$), if
\begin{align}
\sum_{i=1}^k x_i \da &\leq \sum_{i=1}^k y_i \da \quad {\rm for} \; 1 \leq k \leq d \\
\sum_{i=1}^d x_i \da &= \sum_{i=1}^d y_i \da \; .
\end{align}
We can now recall the following lemma on eigenvalues of Hermitian matrices, which can be found, e.g., in~\cite[Problem III.6.14]{bhatia97}:
\begin{lemma}[\citet{bhatia97}]
  \label{lem.logmaj}
  Let $X, Y$ be HPD matrices. Then 
  \begin{equation*}
    \lambda^\da(X) \cdot \lambda^\ua(Y) \prec \lambda (XY) \prec \lambda^\da(X) \cdot \lambda^\da(Y),
  \end{equation*}
  where $\lambda^\da(X)$ ($\lambda^\ua$) denote eigenvalues of $X$ arranged in decreasing (increasing) order, $\prec$ denotes the majorization order and $\cdot$ denotes the elementwise product. If $X, Y$ are Hermitian, we have
  \begin{equation*}
  \ip{\lambda^\da (X)}{\lambda^\ua (Y)} \leq \trace (XY) \leq \ip{\lambda^\da (X)}{\lambda^\da (Y)} \; ,
  \end{equation*}
  with equality, if the the product $XY$ is symmetric.
\end{lemma}

\begin{proof}(Theorem~\ref{thm:logtrace})
  First, introduce the variable $Y=XZX$; then \eqref{eq.12.R} becomes
\begin{equation}
  \label{eq:6-rfw}
  \min_{L' \preceq Y \preceq U'}\quad \trace(S\log Y),
\end{equation}
where the constraints have also been modified to $L' \preceq Y \preceq U'$, where $L'=XLX$ and $U'=XUX$. Diagonalizing $S$ as $S=QDQ^*$, we see that $\trace(S\log Y) = \trace(D\log W)$, where $W=Q^*YQ$. Thus, instead of~\eqref{eq:6-rfw} it suffices to solve
\begin{equation}
  \label{eq:7-rfw}
  \min_{L'' \preceq W \preceq U''}\quad\trace(D\log W),
\end{equation}
where $L''=Q^*L'Q$ and $U''=Q^*U'Q$. We have strict inequality $U'' \succ L''$, thus, our constraints are $0 \prec W-L'' \preceq U''-L''$, which we may rewrite as $0 \prec R \preceq I$, where $R = (U''-L'')^{-\nfh}(W-L'')(U''-L'')^{-\nfh}$. Notice that
\begin{equation*}
  W = (U''-L'')^{\nfh}R(U''-L'')^{\nfh} + L''.
\end{equation*}
Thus, problem~\eqref{eq:7-rfw} now turns into
\begin{equation}
\label{eq:12-rfw}
  \min_{0 \prec R \preceq I }\quad\trace(D\log (P^*R P+L'')) \; ,
\end{equation}
where $U'' - L''=P^* P$. To minimize the trace, we have to minimize the sum of the eigenvalues of the matrix term. Using Lemma~\ref{lem.logmaj} we see that
\begin{equation*}
\trace(D \log(P^* R P+L'')) = \lambda^\da (D)  \cdot \lambda^\da \left(	 \log (P^* R P + L'') \right) \; .
\end{equation*}
Noticing that $D$ is diagonal, we have to consider two cases:
\begin{enumerate}
\item If $d_{ii}>0$, the corresponding element of $\lambda^\da(\log(P^* R P +L''))$ should be minimized.
\item If $d_{ii} \le 0$, the corresponding element of $\lambda^\da(\log(P^* R P+L''))$ should be maximized.
\end{enumerate}
Note that the matrix logarithm $\log(\cdot)$ and the map
$R \mapsto P^* R P +L''$ are operator monotone.  Now, without loss of generality, assume that $R$ is diagonal and recall that, by construction $0 \prec R \preceq I$. Then, setting
\begin{equation}
  \label{eq:11-rfw}
  r_{ii} =
  \begin{cases}
    0 & d_{ii} \geq 0\\
    1 & d_{ii} < 0.
  \end{cases}
\end{equation}
achieves the respective minimum and maximum in cases (i) and (ii) due to operator monotonicity.

Thus, we see that $Y=Q(P^* R P+L'')Q^*=Q(P^* \left[ -\sgn(D)\right]_+ P+L'')Q^*$, and we immediately obtain the optimal $Z=X^{-1}YX^{-1}$.
\end{proof}

\begin{rmk}\normalfont
Computing the optimal direction $Z_k$ takes one Cholesky factorization, two matrix square roots (Schur method), eight matrix multiplications and one eigenvalue decomposition. This gives a complexity of $O(N^3)$. On our machines, we report $\approx \frac{1}{3} N^3 + 2 \times 28.\bar{3} N^3 + 8 \times 2N^3 + 20 N^3 \approx 93 N^3$. %, which is about three times larger than the complexity of the Euclidean linear oracle.
\end{rmk}

\iffalse
%{
\subsubsection*{Means of Hermitian Matrices}
In analogy to the well-known concept of means for sets of numbers, one can define means for sets of matrices:
%
\begin{defn}[Matrix mean, \cite{bhatia07}]
\label{def:mean}
Let $X,Y \in \H_N$ and ``$\leq$'' the L{\"o}wner order on $\H_N$. A \emph{mean} is defined as a binary operation
\[	(X,Y) \longmapsto M(X,Y)\]
which fullfills the following conditions:
\begin{enumerate}
\item $M(X,Y)>0$,
\item $X \leq Y \; \Rightarrow \; X \leq M(X,Y) \leq Y$,
\item $M(X,Y)=M(Y,X)$ (\emph{symmetry}),
\item $M(X,Y)$ monotone increasing in $X$ and $Y$,
\item $M(Z^{*}XZ,Z^{*}YZ)=Z^{*}M(X,Y)Z$ for $Z$ non-singular (\emph{congruence invariance}).
\end{enumerate}
\end{defn}

It is easy to see that translating the notions for the \emph{arithmetic mean} 
\begin{align*}
A(X,Y)=\frac{1}{2}(X+Y)
\end{align*}
 and the \emph{harmonic mean}
\begin{align*}
H(X,Y)=\left( \frac{X^{-1}+Y^{-1}}{2}	\right)^{-1}
\end{align*}
to matrices yields operations that fullfill (i)-(v). However, a direct matrix analogue of the \emph{geometric} mean containing $X^{\frac{1}{2}}Y^{\frac{1}{2}}$ would not lie in $\H_N$ and would, in particular, violate the positivity condition (i) unless $X$ and $Y$ are commuting. Therefore, a more sophisticated notation is necessary to extend the concept of the geometric mean onto matrix spaces.
%}
\fi
 
\subsection{Application to the Riemannian mean}
\subsubsection{Computing the Riemannian mean}
\label{sec:mgm.real}
Statistical inference frequently involves computing averages of input quantities. Typically encountered data lies in Euclidean spaces where arithmetic means are the ``natural'' notions of averaging. However, the Euclidean setting is not always the most natural or efficient way to represent data. Many applications involve non-Euclidean data such as graphs, strings,  or matrices \citep{le2001diffusion,nieBha13}. In such applications, it is often beneficial to represent the data in its natural space and adapt classic tools to the specific setting. In other cases, a problem might be very hard to solve in Euclidean space, but may become more accessible when viewed through a different geometric lens. 

This section considers one of the later cases, namely the problem of determining the \emph{geometric matrix mean} (\emph{Karcher mean problem}). While there exists an intuitive notion for the geometric mean of sets of positive real numbers, this notion does not immediately generalize to sets of positive definite matrices due to the lack of commutativity on matrix spaces. Over a collection of Hermitian, positive definite matrices the geometric mean can be viewed as the geometric optimization problem
\begin{equation}
  \label{eq.1}
  G := \argmin_{X \succ 0}\quad\left[\phi(X) = \nlsum_{i=1}^m w_i\riem^2(X,A_i)\right] \; ,
\end{equation}
where $\delta_R$ denotes the Riemannian metric.
In an Euclidean setting, the problem is \emph{non-convex}. However, one can view Hermitian, positive matrices as points on a Riemannian manifold and compute the geometric mean as the Riemannian centroid.  The corresponding optimization problem (Eq. \ref{eq.1}) is 
\emph{geodesically convex}~\citep{PALFIA2016951}.  In this section we look at the problem through both geometric lenses and provide efficient algorithmic solutions while illustrating the benefits of switching geometric lenses in geometric optimization problems.

There exists a large body of work on the problem of computing the geometric matrix means~\citep{gm_survey}. Classic algorithms like \emph{Newton's method} or \emph{Gradient Decent} (\emph{GD}) have been successfully applied to the problem. Standard toolboxes implement efficient variations of \emph{GD} like \emph{Steppest Decent} or \emph{Conjugate Gradient} (\emph{Manopt} ~\citep{manopt}) or Richardson-like linear gradient decent (\emph{Matrix Means Toolbox}~\citep{mmtoolbox}). Recent work by Yuan et al.~\citep{YHAG2017} analyzes condition numbers of Hessians in Riemannian and Euclidean steepest-decent approaches that provide theoretical arguments for the good performance of Riemannian approaches.

Recently, T. Zhang developed a majorization-minimization approach with asymptotic linear convergence~\citep{zhang2017}. In this section we apply the above introduced Riemannian version of the classic \emph{conditional gradient} method by M. Frank and P. Wolfe~\citep{fw_original} (\rfw) for computing the geometric matrix mean in Riemannian settings with linear convergence. Here, we exploit the geodesic convexity of the problem. To complement this analysis, we show that recent advances in non-convex analysis~\citep{LJ16} can be used to develop a Frank-Wolfe scheme (\efw) for the non-convex Euclidean case (see Appendix~\ref{sec:e-lo}).

 In the simple case of two PSD matrices $X$ and $Y$ one can view the geometric mean as their metric mid point computed by \cite{KuboAndo}
\begin{equation}
G(X,Y)= X \#_t Y = X^{\frac{1}{2}} \left( X^{-\frac{1}{2}} Y X^{-\frac{1}{2}}	\right)^t X^{\frac{1}{2}} \; .
\end{equation}
More generally, for a collection of $M$ matrices, the geometric mean can be seen as a minimization problem of the sum of squares of distances \cite{BhatiaHolbrook},
\begin{equation}
G(A_1, ..., A_M)= \argmin_{X \succ 0} \sum_{i=1}^M \delta_R^2 (X, A_i) \; ,
\end{equation}
with the Riemannian distance function 
\begin{equation}
d(X,Y)=\norm{\log(X^{-\frac{1}{2}} Y X^{-\frac{1}{2}})} \; .
\end{equation}
Here we consider the more general \emph{weighted} geometric mean:
\begin{equation}
G(A_1, ..., A_M)= \argmin_{X \succ 0} \sum_{i=1}^M w_i \delta_R^2 (X, A_i) \; .
\end{equation}
E. Cartan showed in a Riemannian setting that a global minimum exists, which led to the term \emph{Cartan mean} frequently used in the literature. In this setting, one can view the collection of matrices as points on a Riemannian manifold. H. Karcher associated the minimization problem with that of finding centers of masses on these manifolds \citep{karcher}, hence motivating a second term to describe the geometric matrix mean (\emph{Karcher mean}).

The geometric matrix mean enjoys several key properties. We list below the ones of crucial importance to our paper and refer the reader to~\cite{LL14,lim.palfia} for a more extensive list. To state these results, we recall the general form of the two other basic means of operators: the (weighted) harmonic and arithmetic means, denoted by $H$ and $A$ respectively.

\begin{equation*}
  H := \left(\nlsum_{i=1}^M w_i \inv{A_i}\right)^{-1},\qquad A := \nlsum_{i=1}^M w_i A_i.
\end{equation*}

 Then, one can show the following well-known operator inequality that relates $H$, $G$, and $A$:
\begin{lemma}[Means Inequality, \cite{bhatia07}]
  \label{prop.key}
  Let $A_1,\ldots,A_M > 0$, and let $H, G$, and $A$ denote their (weighted) harmonic, geometric, and arithmetic  means, respectively. Then,
  \begin{equation}
  H \preceq G \preceq A \; .
  \end{equation}
\end{lemma}

The key computational burden of all our algorithms lies in computing the gradient of the objective function~\eqref{eq.1}. A short calculation (see e.g., \cite[Ch.6]{bhatia07}) shows that if $f(X) = \riem^2(X,A)$, then $\nabla f(X) = \inv{X}\log(XA^{-1})$. Thus, we immediately obtain
\begin{equation*}
  \nabla\phi(X) = \nlsum_i w_i\inv{X}\log(XA_i^{-1}).
\end{equation*}

\subsubsection{Implementation}
We compute the \emph{geometric matrix mean} with Algorithm~\ref{alg.rfw}. For the PSD manifold, line 3 can be written as
\begin{equation}
\label{eq:psd-oracle}
Z_k \gets \argmin_{H \preceq Z \preceq A}\ip{X_k^{\nfh}\nabla \phi(X_k)X_k^{\nfh}}{\log(X_k^{-\nfh}ZX_k^{-\nfh})} \; .
\end{equation}
Note that the operator inequality  $H \preceq G \preceq A$ given by Lemma~\ref{prop.key} plays a crucial role: It shows that the optimal solution lies in a compact set so we may as well impose this compact set as a constraint to the optimization problem (i.e., we set $\Xc=\lbrace H \preceq Z \preceq A	\rbrace$). 
We implement the linear oracles as discussed above: In the Euclidean case, a closed-form solution is given by
\begin{equation}
  \label{eq.17}
  Z = H + P^*Q[-\sgn(\Lambda)]_{+}Q^*P,
\end{equation}
where $A-H = P^*P$ and $P\nabla\phi(X_k)P^* = Q\Lambda Q^*$. Analogously, for the Riemannian case, the ``linear'' oracle
\begin{equation}
  \min_{H \preceq Z \preceq A}\quad \ip{\nabla \phi(X_k)}{\log(Z)} \; ,
\end{equation}
is well defined and solved by
\begin{equation}
Z=Q\left( P^*[-\sgn(\Lambda)]_{+} P + \hat{H} \right) Q^{*} \; ,
\end{equation}
with $\hat{A}-\hat{H}=P^{*}P$, $\hat{A}=Q^* A Q$, $\hat{H}=Q^{*}HQ$ and $\nabla\phi(X_k) = Q\Lambda Q^*$ (see Prop.~\ref{eq:psd-oracle}). The resulting Frank-Wolfe method for the geometric matrix mean is given by Algorithm~\ref{alg.means}.
\begin{algorithm}[H]
  \caption{FW for fast Geometric mean (GM)/ Wasserstein mean (WM)}
  \label{alg.means}
  \begin{algorithmic}[1]
     \State $(A_1,\ldots,A_N)$, $\bm{w} \in \rplus^N$
     \State $\bar{X} \approx \argmin_{X > 0} \nlsum_iw_i \riem^2(X,A_i)$
     \State $\beta = \min_{1 \leq i \leq N} \lambda_{{\rm min}} (A_i)$
     \For {$k=0,1,\dots$}
     	\State Compute gradient.
        \State \hskip12pt GM: $\nabla\phi(X_k) = X_k^{-1}\bigl(\nlsum_iw_i \log(X_kA_i^{-1})\bigr)$
        \State \hskip12pt WM: $\nabla\phi(X_k) =  \sum_i w_i \left(I - (A_i X_k )^{-1/2} A_i \right)$
        \State Compute $Z_k$:
        \State  \hskip12pt GM: $Z_k \gets \argmin_{H \preceq Z \preceq A}\ip{X_k^{1/2}\nabla
           \phi(X_k)X_k^{1/2}}{\log(X_k^{-1/2}ZX_k^{-1/2})}$
         \State \hskip12pt WM: $Z_k \gets \argmin_{\beta I \preceq Z \preceq A}\ip{X_k^{1/2}\nabla
           \phi(X_k)X_k^{1/2}}{\log(X_k^{-1/2}ZX_k^{-1/2})}$
        \State Let $\alpha_k \gets \frac{2}{k+2}$.
        \State Update $X$:
        \State \hskip12pt $X_{k+1} \gets X_k \gm_{\alpha_k}Z_k$
     \EndFor
     \State \textbf{return} $\bar{X}=X_k$
   \end{algorithmic}
 \end{algorithm}
 \subsection{Application to Bures-Wasserstein barycenters}
 \subsubsection{Computing Bures-Wasserstein barycenters on the Gaussian density manifold}
 \label{sec:Bures}
As a second application of \rfw, we consider the computation of Bures-Wasserstein barycenters of multivariate (centered) Gaussians. This application is motivated by optimal transport theory; in particular, the Bures-Wasserstein barycenter is the solution to the multi-marginal transport problem~\citep{Malago,bhatia2}: Let $\lbrace A_i \rbrace_{i=1}^m \in \mathbb{P}_d$ and $w=(w_1, ..., w_m)$ a vector of weights ($\sum_i w_i =1; \; w_i >0, \forall i$). The minimization task
 \begin{align}\label{eq:wm-obj}
 \min_{\substack{X \in \mathbb{P}_d \\ X \succ 0}} \sum_{i=1}^m w_i d_W^2(X,A_i)
 \end{align}
 is called \emph{multi-marginal transport problem}. Its unique minimizer
 \begin{align*}
 \Omega(w; \lbrace A_i \rbrace) = \argmin_{X \succ 0} \sum_{i=1}^m w_i d_W^2 (X,A_i) \in \mathbb{P}_d
 \end{align*}
 is the \emph{Bures-Wasserstein barycenter}. Here, $d_W$ denotes the \emph{Wasserstein distance}
 \begin{align}\label{eq:w-dist}
 d_W(X,Y) = \bigl[{\rm tr} (X+Y) - 2{\rm tr} \bigl(	X^{1/2} Y X^{1/2}	\bigr)^{1/2} 	\bigr]^{1/2}  \; .
\end{align}
To see the connection to optimal transport, note that Eq.~\ref{eq:wm-obj} corresponds to a least squares optimization over a set of multivariate center Gaussians with respect to the Wasserstein distance. The Gaussian density manifold is isomorphic to the manifold of symmetric positive definite matrices, which allows for a direct application of \rfw and the setup in the previous section to Eq.~\ref{eq:wm-obj}, albeit with a different set of constraints. In the following, we discuss a suitable set of constraint and adapt \rfw for the computation of the Bures-Waserstein barycenter (short: Wasserstein mean).

First, note that as for the Karcher mean, one can show that the Wasserstein mean of two matrices is given in closed form; namely as
\begin{align}\label{eq:wm-gmap}
X \lozenge_t Y =  (1-t)^2 X + t^2 Y + t(1-t) \left( (XY)^{1/2} + (YX)^{1/2} 	\right) \; .
\end{align}
However, the computation of the barycenter of $m$ matrices ($m>2$) requires solving a quadratic optimization problem. Note, that Eq.~\ref{eq:wm-gmap} defines a geodesic map from $X$ to $Y$.
Unfortunately, an analogue of the \emph{means inequality} does not hold in the Wasserstein case. However, one can show, that the arithmetic matrix mean always gives an upper bound, providing one-sided constraints~\citep{bhatia}:
\begin{align*}
A_i \lozenge_t A_j \preceq \frac{A_i +A_j}{2} \; ,
\end{align*}
and similarly for the arithmetic mean $A(w; \lbrace A_i \rbrace)$ of $m$ matrices. Moreover, one can show that $\alpha I \preceq X$~\cite{bhatia2}, where $\alpha$ is the minimal eigenvalue over $\lbrace A_i \rbrace_{i=1}^m$, i.e.
\begin{align*}
\alpha = \min_{1 \leq j \leq m} \lambda_{{\rm min}} (A_j) \; .
\end{align*}
In summary, this gives the following constraints on the Wasserstein mean:
\begin{align}\label{eq:wm-constraints}
\alpha I \preceq \Omega(w; \lbrace A_i \rbrace) \preceq A(w; \lbrace A_i \rbrace) \; .
\end{align}

Next, we will derive the gradient of the objective function. According to Eq.~\ref{eq:wm-obj}, the objective is given as
\begin{equation}
  \label{eq:phi-dw}
\phi(X) := \nlsum_j w_j \bigl[	\trace(A_j) + \trace(X) - 2 \trace \bigl(A_j^{1/2} X A_j^{1/2} \bigr)^{1/2}		\bigr]^{1/2} \; .
\end{equation}
Two expressions for the gradient of (\ref{eq:phi-dw}) are derived in the following.
\begin{lemma}
  Let $\phi$ be given by~\eqref{eq:phi-dw}. Then, its gradient is
\begin{equation}
\nabla \phi (X, \mathcal{A}) = \sum_j \omega_j \left( I - A_j \#	X^{-1}	\right) \; .
\end{equation}
\end{lemma}
\begin{proof}
\begin{align*}
\nabla \phi(X, \mathcal{A}) &= \underbrace{\frac{d}{dX} \sum_j w_j \trace{A_j}}_{{\rm vanishes}} + \frac{d}{dX} \sum_j w_j \left(	\trace(X) - 2 \trace \left(A_j^{1/2} X A_j^{1/2}	\right)^{1/2} \right) \\
&= \sum_j w_j \frac{d}{dX}  \trace(X) -  \sum_j w_j 2 \frac{d}{dX}  \trace \left(A_j^{1/2} X A_j^{1/2}	\right)^{1/2} \; .
\end{align*}
Note that $\frac{d}{dX}  \trace(X) = I$. Consider
\begin{align*}
2 \frac{d}{dX} \trace \left(	A^{1/2} X A^{1/2}	\right)^{1/2} 
&\overset{(1)}{=} 2 \frac{d}{dX} \left[	
\left( \left(	A^{1/2} X A^{1/2}	\right)^{1/2} \right)^* \left(\left(	A^{1/2} X A^{1/2}	\right)^{1/2} \right)
\right]^{1/2} \\
&\overset{(2)}{=} 2 \frac{d}{dX} \left(	A^{1/2} X A^{1/2}	\right)^{1/2} \\
&= A^{1/2}  \left(	A^{1/2} X A^{1/2}	\right)^{-1/2} A^{1/2}  \\
&= A^{1/2}  \left(	A^{-1/2} X^{-1} A^{-1/2}	\right)^{1/2} A^{1/2}  \\
&= A \# X^{-1}  \; ,
\end{align*}
where (1) follows from the chain rule and (2) from $A$ and $X$ being symmetric and positive definite: If $A$ and $X$ are symmetric and positive definite, then $A$ has a unique positive definite root $A^{1/2}$ and therefore
\begin{align*}
\left(	A^{1/2} X A^{1/2}	\right)^* = \left(	{A^{1/2}}^* X^* {A^{1/2}}^*	\right) = A^{1/2} X A^{1/2} \; .
\end{align*}
Putting everything together,  the desired expression follows as
\begin{equation}
\nabla \phi(X, \mathcal{A}) 
= \sum_j \omega_j \left( I - A_j \#	X^{-1}	\right) \; .
\end{equation}
\end{proof}

 \subsubsection{Implementation}
Given a set of constraints and an expression for the gradient, we can solve Eq.~\ref{eq:wm-obj} with \textsc{Rfw} using the following setup: 
We write the log-linear oracle as
\begin{align}\label{eq:wm-oracle}
Z_k \in \argmin_{\alpha I \preceq Z \preceq A} \langle \grad \phi(X), \Exp_X^{-1}(Z) \rangle \; ,
\end{align}
where $\grad \phi(X)$ is the Riemannian gradient of $\phi$ and $\Exp_X^{-1}(Z) $ the exponential map with respect to the Riemannian metric. The constraints are given by Eq.~\ref{eq:wm-constraints}. Since the constraint set is given by an interval of the form $L \preceq Z \preceq U$, the oracle can be solved in closed form using Theorem~\ref{thm:logtrace}.
The resulting algorithm is given in Algorithm~\ref{alg.means}.

%%%%%%%%%%%%%%%%%%%%%%%%%%
\section{Specializing RFW to $SO(n)$}
\label{sec:SO(d)}
We have seen in the previous section, that \rfw can be implemented efficiently, whenever the Riemannian ``linear'' oracle can be solved in closed form. In this section we want to show that this is not limited to the case of positive definite matrices. In particular, we specialize \rfw to the \emph{special orthogonal group}
\begin{align*}
SO(n) = \lbrace	 X \in \reals^{n \times n}: X^* X = I_n, \; \det(X)=1	\rbrace 
\end{align*}
and show that here too, we can implement \rfw efficiently.

\subsection{Geometry of $SO(n)$}\label{sec:SO(n)-LO}
We first recall some facts on the geometry of $SO(n)$; for a more comprehensive overview see, e.g.,~\citep{edelman98}. The exponential map and its inverse on $SO(n)$ are given as
\begin{align*}
\Exp_X(Z) &= X^{-1}Z \\
\Exp_X^{-1}(Z) &= \log (X^{-1}Y) \; .
\end{align*}
Note that the tangent space of the manifold is given by the space of skew-symmetric matrices
\begin{align*}
T_X SO(n) = \lbrace	 X \in \reals^{n \times n}: \; X=-X^*	\rbrace \; .
\end{align*}
The standard Riemannian metric on the special orthogonal group (and more generally on the orthogonal group $O(n)$) is given by
\begin{align*}
g_c(\Delta, \Delta) = \trace \left[ \Delta^* \left(	I - \frac{1}{2} Y^*Y	\right) \Delta	\right] \; ,
\end{align*}
where $\Delta = YA + Y_{\perp}B$ are in $T_X SO(n)$. With respect to this metric, a geodesic from $Y \in SO(n)$ and in direction $H \in T_Y O(n)$ is given by~\citep{edelman98}
\begin{align}\label{eq:so-map}
\gamma(Y, H; t) = Y M(t) + Q N(t) \; ,
\end{align}
where $QR=(I - Y^*Y)H$ can be computed via QR-decomposition and 
\begin{align*}
\binom{M(t)}{N(t)}	 = \exp t \begin{pmatrix}
A & -R^* \\
R & 0
\end{pmatrix} \binom{I_n}{0} \; .
\end{align*}

\subsection{Solving the Riemannian ``linear'' oracle}
The Riemannian ``linear'' oracle (line 3 in Algorithm~\ref{alg.rfw}) for the special orthogonal group is given by
\begin{equation}
  \label{eq:SO-oracle}
  \min_{Z \in SO(n)}\quad \ip{\grad \; \phi(X)}{\log(X^{-1}Z} \; .
\end{equation}
We can rewrite this as the trace maximization
\begin{align*}
\min_{\substack{Z^* Z = I_n \\ \det(Z)=1 }} \; \trace \ip{G}{\log(X^{-1} Z)} \; ,
\end{align*}
where $G$ denotes the Riemannian gradient $\grad \; \phi(X)$. With a similar argument as in Theorem~\ref{thm:logtrace}, this can be solved in closed form:
\begin{theorem}
The solution to the optimization problem
\begin{align*}
\min_{\substack{Z^* Z = I_n \\ \det(Z)=1 }} \; \trace \ip{G}{ \log(X^{-1} Z)} \; ,
\end{align*}
is given by $Z = X V U^*$ where $G = U D V^*$ is a singular value decomposition of the gradient.
\end{theorem}
\begin{proof}
Consider a singular value decomposition $G=UDV^*$ with $U,V$ unitary. Furthermore, introduce the shorthand $Y=X^{-1}Z$.  We can then rewrite the oracle as follows:
\begin{align*}
\trace \left(	G \log Y 	\right) &= \trace \left( U D V^* \log(Y) \right) \\
&= \trace \left( D V^* \log(Y) U \right) \\
&=^{\dagger} \trace \left( D  \log(V^* Y U ) \right) \; .
\end{align*}
where ($\dagger$) follows from $V, U$ being unitary. Let $W = V^* Y U$. Note that
\begin{align*}
\trace \left( D \log(W) \right) = \ip{D}{\log(W)} \leq \norm{D}_2 \cdot \norm{\log(W)}_2 \; .
\end{align*}
Since $\norm{X}_2 = \sqrt{\trace X^* X}=\left( \sum_{i}	\sigma_i(X)	\right)^{1/2}$, the right hand side is maximized, if the singular values of $W$ are maximal. Since $W \in O(k)$ its singular values are $\sigma_i=1$. The maximum is therefore achieved by setting $W=I$. This gives
\begin{align*}
I = W = V^* Y U = V^* \left( X^{-1} Z \right) U \; .
\end{align*} 
Solving for $Z$ gives a solution for maximizing the original problem:
\begin{align*}
Z=XVU^* \; .
\end{align*}
\end{proof}
\begin{rmk}[Nonconvex constraints]\normalfont
Note, that in this setting, the constraint set $\Xc$ is not g-convex. However, we can still apply \rfw, since the geodesic update (\ref{eq:so-map}) remains in the feasible region. To see this, note that $SO(n)$ is a compact submanifold of $O(n)$. In particular, the manifold of orthogonal matrices has two compact connected components (corresponding to matrices with $\det(X)=1$, i.e., $SO(n)$, and $\det(X)=-1$, respectively). Hence, when initialized at some $X \in SO(n)$, the geodesic update (\ref{eq:so-map}) will always remain within $SO(n)$. 
\end{rmk}

\subsection{Application to the Procrustes problem}
The \emph{Procrustes problem} is a classic optimization task that seeks to compute a synchronization between two matrices $A, B \in \mathbb{R}^{n \times k}$. In particular, the goal is to find a rotation matrix $X \in SO(n)$ that achieves the best possible map between data matrices $A$ and $B$. The corresponding objective function is
\begin{align*}
f(X) = \frac{1}{2} \norm{XB - A}^2 = \frac{1}{2} \norm{A}^2 + \frac{1}{2} \norm{B}^2 - \trace \left(B^T X^T A	 \right) \; .
\end{align*}
The optimal rotation $X$ can be then be found by solving
\begin{align*}
\min_{X \in SO(n)} \; \left[ \phi(X) := - \trace \left( X^T A B^T \right) \right] \; .
\end{align*}
The Euclidean gradient of the objective is easily derived as
\begin{align*}
\nabla \phi (X) = - AB^T, 
\end{align*}
while the Riemannian gradient is given by
\begin{align*}
\grad \; \phi(X) = X \cdot {\rm skew} \left(	X^T \nabla \phi(X)	\right) \; .
\end{align*}
Recall that ${\rm skew} (X) = \frac{1}{2} \left(X - X^T \right)$. The resulting algorithm is given in Algorithm~\ref{alg.procrustes}.
\begin{algorithm}[H]
  \caption{RFW for Procrustes problem}
  \label{alg.procrustes}
  \begin{algorithmic}[1]
     \State Input: $A, B \in \mathbb{R}^{n \times k}$.
     \State Assume access to the geodesic map $\gamma(X,Z;\alpha)$ (Eq.~\ref{eq:so-map}).
     \State $\bar{X} \approx \argmin_{X \in SO(n)} - \trace \left( X^T A B^T \right)$
     \State Precompute $\nabla \phi (X) = - AB^T$.
     \For {$k=0,1,\dots$}
      \State Compute gradient: $\grad \; \phi(X_k) = X_k \cdot {\rm skew} \left(	X_k^T \nabla \phi(X_k)	\right)$.
       \State Compute $Z_k$ via Riemannian ``linear'' oracle:
        \State  \hskip12pt $ Z_k \gets \argmin_{Z \in T_{X_k}SO(n)} \; \ip{\grad \; \phi(X_k)}{ \log(X_k^{-1} Z_k)}$
        \State Let $\alpha_k \gets \frac{2}{k+2}$.
        \State Update $X_k$ as
        \State \hskip12pt $X_{k+1} \gets \gamma(X_k, Z_k; \alpha_k)$.
     \EndFor
     \State \textbf{return} $\bar{X}=X_k$.
   \end{algorithmic}
 \end{algorithm}
 %

%%%%%%%%%%%%%%%%%%%%%%%%%%
 \section{Approximately solving the Riemannian ``linear'' oracle}
In the previous two sections we have given examples of applications where the Riemannian ``linear'' oracle can be solved in closed form -- rendering the resulting \rfw algorithm into a practical method.  Unfortunately, the ``linear'' oracle  is in general a nonconvex subroutine. Therefore, it can be challenging to find efficient solutions for constrained problems in practise.

One remedy for such situations is to solve the ``linear'' oracle only approximately. In the following, we show that we can recover sublinear convergence rates for \rfw, even if we solve the ``linear'' oracle only approximately.  This extension greatly widens the range of possible applications for \rfw. For instance,  while we currently do not have closed form solutions for the``linear'' oracle in some of the examples in section~\ref{sec:constrained-examples}, we could could find approximate solutions via relaxations or iteratively solving the respective subroutine.

%\subsection{Convergence analysis}
In the following, we say that $Z' \in\Xc$ is a \emph{$\delta$-approximate linear minimizer}, if
\begin{align*}
\ip{\grad \; \phi(X_k)}{\Exp_{X_k}^{-1}(Z')} \leq \min_{Z \in \Xc} \ip{\Exp_{X_k}^{-1}(Z)}{\grad \; \phi(X_k)} + \frac{1}{2} \delta \eta M_{\phi} \; .
\end{align*}
We want to show the following sublinear convergence guarantee:
\begin{theorem}\label{thm:conv-delta}
 Let $X^*$ be a minimum of a geodesically convex function $\phi$ and $\delta \geq 0$ the accuracy to which the ``linear'' oracle is solved in each round. Then, the sequence of iterates $X_k$ generated by Algorithm~\ref{alg.rfw} satisfies  
 \begin{align*}
 \phi(X_k)-\phi(X^*) \leq \frac{2M_{\phi}}{k+2} (1+\delta) \; .
 \end{align*}
\end{theorem}
The proof relies on adapting proof techniques from~\citep{clarkson,jaggi2013revisiting} to the Riemannian setting.  It utilizes the following auxiliary lemma:
\begin{lemma}\label{lem:smooth-delta}
For a steps size $\eta \in (0,1)$ and accuracy $\delta$, we have
\begin{align*}
\phi(X_{k+1}) \leq \phi(X_k) - \eta \ip{\grad \; \phi(X_k)}{\Exp_{X_k}(Z')} + \frac{1}{2}\eta^2 M_{\phi}(1+\delta) \; .
\end{align*}
\end{lemma}
A proof of the Lemma can be found in Appendix~\ref{app:A}.
The proof of Theorem~\ref{thm:conv-delta} follows then from Lemma~\ref{lem:induction} and~\ref{lem:smooth-delta} similar to Theorem~\ref{thm:conv-FWR}.

%%%%%%%%%%%%%%%%%%%%%%%%%%
 \section{Computational Experiments}
 \label{sec:expts}
In this section, we will make some remarks on the implementation of Algorithm~\ref{alg.means} and show numerical results for computing the geometric matrix mean for different parameter choices. To evaluate the efficiency of our method, we compare its performance against a selection of state-of-the-art methods. Additionally, we use Algorithm~\ref{alg.means} to compute Wasserstein barycenters of positive definite matrices.
All computational experiments are performed using \matlab.

\subsection{Computational Considerations}
When implementing the algorithm we can take advantage of the positive definiteness of the input matrices. For example, if using \matlab, rather than computing $X^{-1}\log(XA_i^{-1})$, it is more preferable to compute
\begin{equation*}
  X^{-\nfh}\log(X^{\nfh}A_i^{-1}X^{\nfh})X^{\nfh},
\end{equation*}
because both $X^{-\nfh}$ and $\log(X^{\nfh}A_i^{-1}X^{\nfh})$ can be computed by suitable eigendecomposition. In contrast, computing $\log(XA_i^{-1})$ invokes the matrix logarithm (\texttt{logm} in \matlab), which can be much slower. 

%\subsubsection{Stepsize selection} 
To save on computation time, we prefer to use a diminishing scalar as the stepsize in Algorithm~\ref{alg.means}. In principle, this simple stepsize selection could be replaced by a more sophisticated Armijo-like line-search or even exact minimization by solving
\begin{equation}
  \label{eq.9}
  \alpha_k \gets \argmin_{\alpha \in [0,1]}\quad \phi(X_k + \alpha_k(Z_k-X_k)) \; .
\end{equation}
This stepsize tuning may accelerate the convergence speed of the algorithm, but it must be combined with a more computational intensive strategy of ``away'' steps~\cite{marcotte} to obtain a geometric rate of convergence. However, we prefer Algorithm~\ref{alg.means} for its simplicity and efficiency.

%\subsubsection{Accelerating the convergence}
Theorems~\ref{thm:conv-FWR} and~\ref{thm:conv-FWR-nonconvex} show that Algorithm~\ref{alg.rfw} converges at the global (non-asymptotic) rates $O(1/\epsilon)$ (g-convex \rfw) and $O(1/\epsilon^2)$ (nonconvex \rfw). However, by further exploiting the simple structure of the constraint set and the ``curvature'' of the objective function, we might obtain a stronger convergence result.

\subsection{Numerical Results}
We present numerical results for computing the Karcher and Wasserstein means for sets of positive definite matrices. We test our methods on both well- and ill-conditioned matrices; the generation of sample matrices is described in the appendix. In addition, we present numerical results for solving the Procrustes problem.
  
\subsubsection{Computing the Riemannian Mean}
\label{sec:exp-karcher}
To test the efficiency of our method, we implemented \rfw (Algorithm~\ref{alg.means}) and its Euclidean counterpart \efw (Algorithm~\ref{alg.efw}) in \matlab and compared its performance on computing the geometric matrix mean against related state-of-the-art methods:
\begin{enumerate}
\item \emph{Riemannian L-BFGS} (textsc{R-LBFGS}\footnote{also known as \textsc{LRBFGS}}) is a quasi-Newton method that iteratively approximates the Hessian for evaluating second-order information~\citep{BFGS}. %We use an improved limited-memory version of the method, implemented in \emph{Manopt}.
\item \emph{Riemannian Barzilai–Borwein} (\textsc{BB}) is a first-order gradient-based method for constraint and unconstraint optimization. It evaluates second-order information from an approximated Hessian to choose the stepsize~\citep{BB}. We use the \emph{Manopt} version of the \textsc{RBB}.
\item \emph{Matrix Means Toolbox} (\textsc{MM})~\citep{mmtoolbox} is an efficient \matlab toolbox for matrix optimization. Its implementation of the geometric mean problem uses a Richardson-like iteration of the form
\begin{align*}
X_{k+1} \leftarrow X_k - \alpha X_k \sum_{i=1}^n \log (A_i^{-1} X_k) \; ,
\end{align*}
with a suitable $\alpha>0$.
\item \textsc{Zhang}~\citep{zhang2017} is a recently published majorization-minimization method for computing the geometric matrix mean.
\end{enumerate}
In addition, we compare against classic gradient-based optimization methods (\emph{Steepest Decent} and \emph{Conjugate Gradient}), see Appendix C.

Note that this selection reflects a broad spectrum of commonly used methods for Riemannian optimization. It ranges from highly specialized approaches that are targeted to the Karcher mean (\textsc{MM}), to more general and versatile methods (e.g., \textsc{R-LBFGS}). A careful evaluation should take these implementation differences into account, by comparing not only CPU time, but also the loss with respect to the number of iterations. In addition to reporting the loss with respect to CPU time, we include results with respect to the number of iterations in Appendix C. One should further compare algorithmic features such as the number of calls to oracles and loss functions. We perform and further discuss such an evaluation below.

We generate a set \emph{A} of \emph{M} positive definite matrices of size \emph{N} and compute the geometric mean with \efw and \rfw as specified in Algorithm~\ref{alg.means}. To evaluate the performance of the algorithm, we compute the cost function
\begin{align}
f(X,A)=\sum_{i=1}^M \norm{\log \left(	X^{-\frac{1}{2}} A_i X^{-\frac{1}{2}} \right)}_F ^2 \; ,
\end{align}
after each iteration. We further include an experiment, where we report the gradient norm $\norm{\grad \; f}_F$.

Figure~\ref{fig:param} shows performance comparisons of all methods for different parameter choice, initializations and condition numbers. In a second experiment (Fig.~\ref{fig:init-w} and Fig.~\ref{fig:init-i}) we compared three different initialization: The harmonic mean and a matrix $A_i \in A$, as well as the arithmetic-harmonic mean $\frac{1}{2}(AM + HM)$. The latter is known to be a good approximation to the geometric mean and therefore gives an initialization close to the solution. We observe that \efw/ \rfw outperform \textsc{BB} for all parameter choices and initializations. Furthermore, \efw/ \rfw perform very competitively in comparison with \textsc{R-LBFGS}, \textsc{Zhang's Method} and \textsc{MM}. In particular, if the initialization is not very close to the final solution (e.g., the arithmetic-harmonic mean), \efw/ \rfw improve on the state-of-the-art methods. In a third experiment we compared the accuracy reached by \efw and \rfw with \textsc{R-LBFGS} as the (in our experiments) most accurate state-of-the-art method (Fig.~\ref{fig:gnorm}). We observe that \efw and \rfw reach a medium accuracy fast; however, ultimately R-LBFGS reaches a higher accuracy.

To account for the above discussed inaccuracies in CPU time as the performance measure, we complement our performance analysis by reporting the loss with respect to the number of iterations (see Appendix C). In addition, we include a comparison of the number of internal calls to cost and gradient functions (Figure~\ref{tab:calls}) for all methods. Note that this comparison evaluates algorithmic features and is therefore machine-independent. The reported numbers were obtained by averaging counts over ten experiments with identical parameter choices. We observe that \textsc{Rfw} and \textsc{Efw} avoid internal calls to the gradient and cost function, requiring only a single gradient evaluation per iteration. This results in the acceleration (i.e., faster convergence) observed in the plots.
\begin{figure}[!htbp]
\centering
     \subfloat[Well-conditioned]{\includegraphics[width=5.1cm]{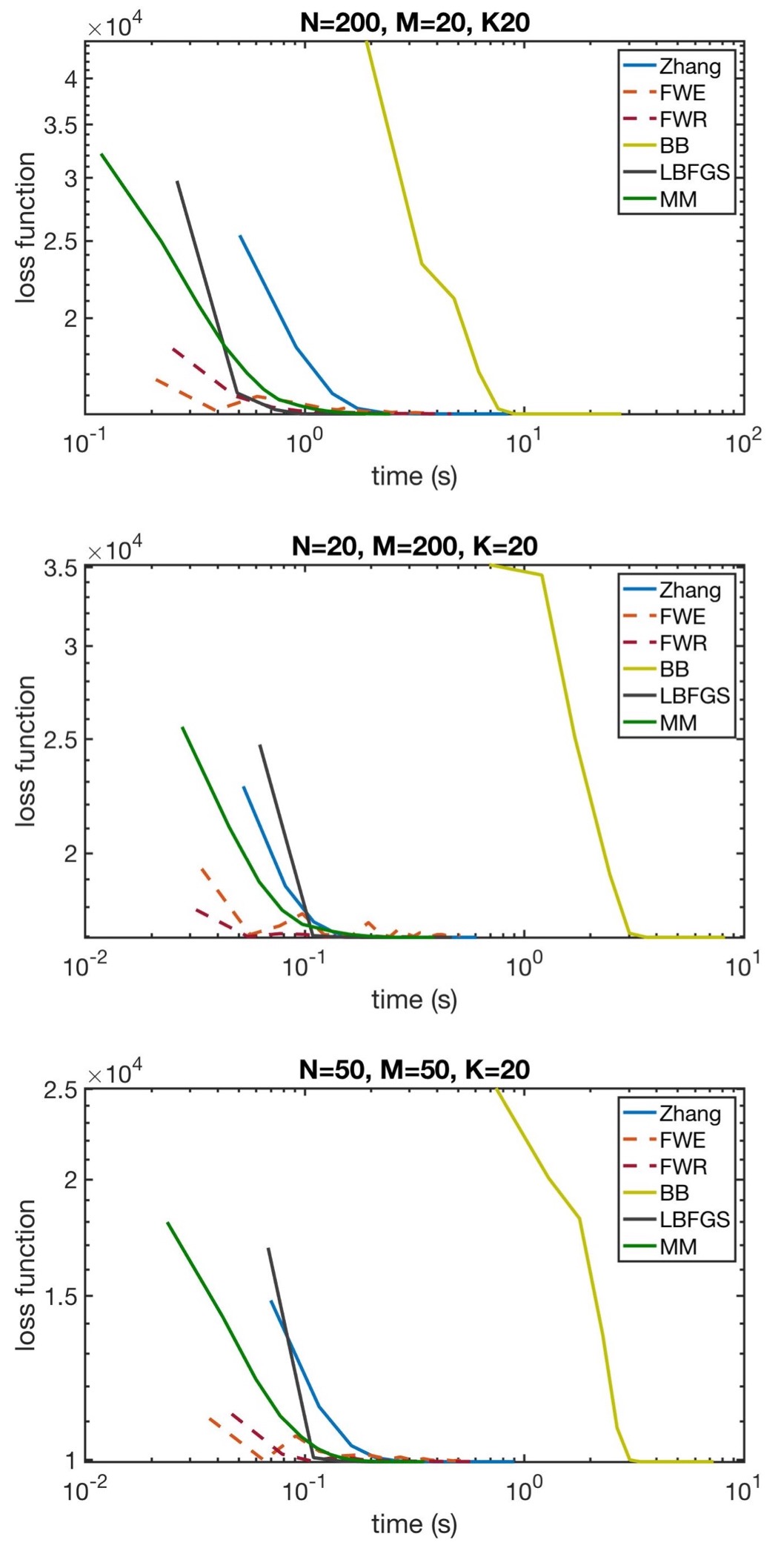}}
     \subfloat[Ill-conditioned]{\includegraphics[width=5cm]{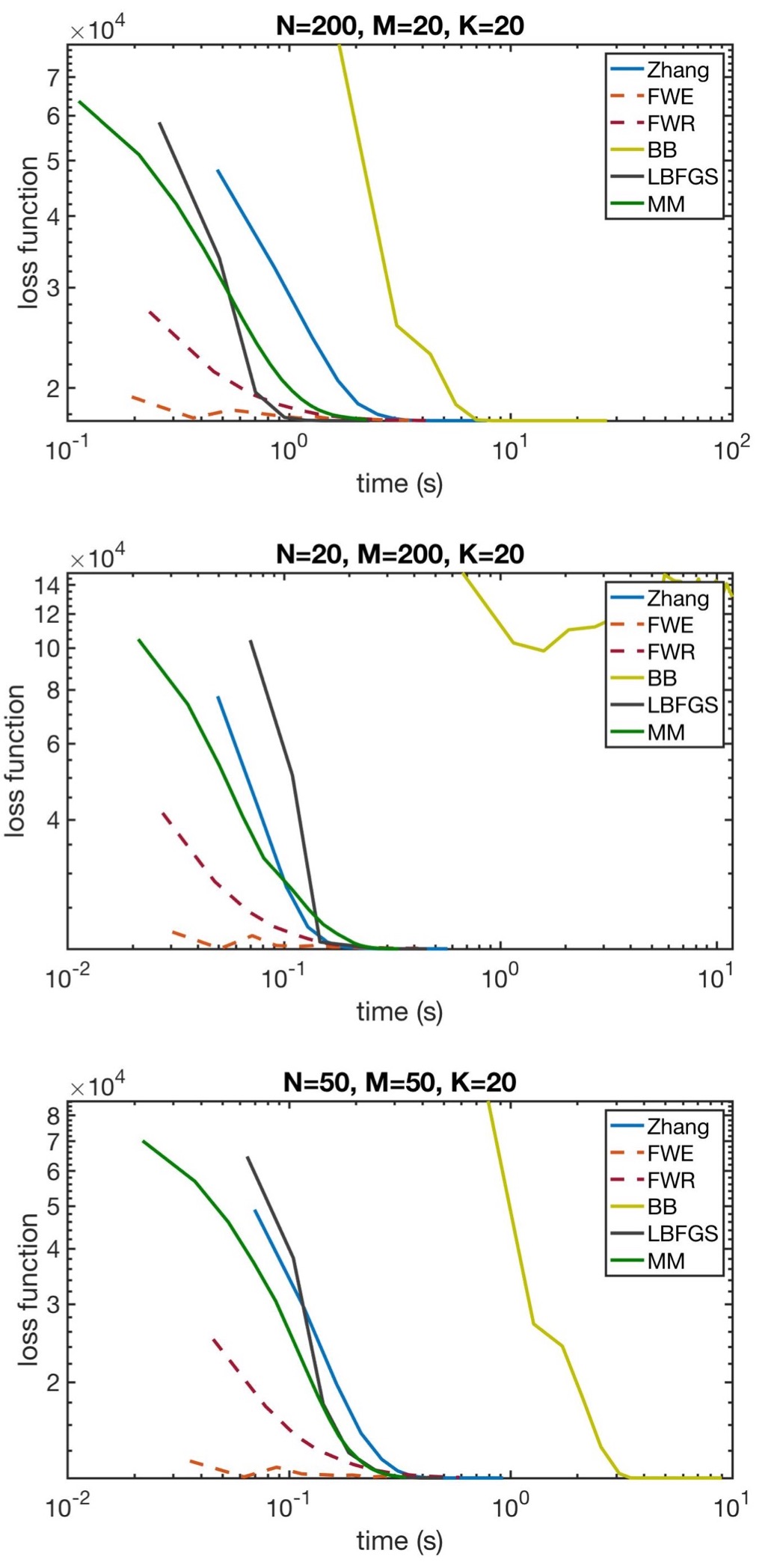}}
     \caption{Performance of \efw and \rfw in comparison with state-of-the-art methods for well-conditioned and ill-conditioned inputs of different sizes ($N$: size of matrices, $M$: number of matrices, $K$: maximum number of iterations). All tasks are initialized with the harmonic mean $x_0=HM$.}
     \label{fig:param}
\end{figure}
\begin{figure}[!htbp]
     \subfloat[$x_0=A_1$]{\includegraphics[width=0.35\textwidth]{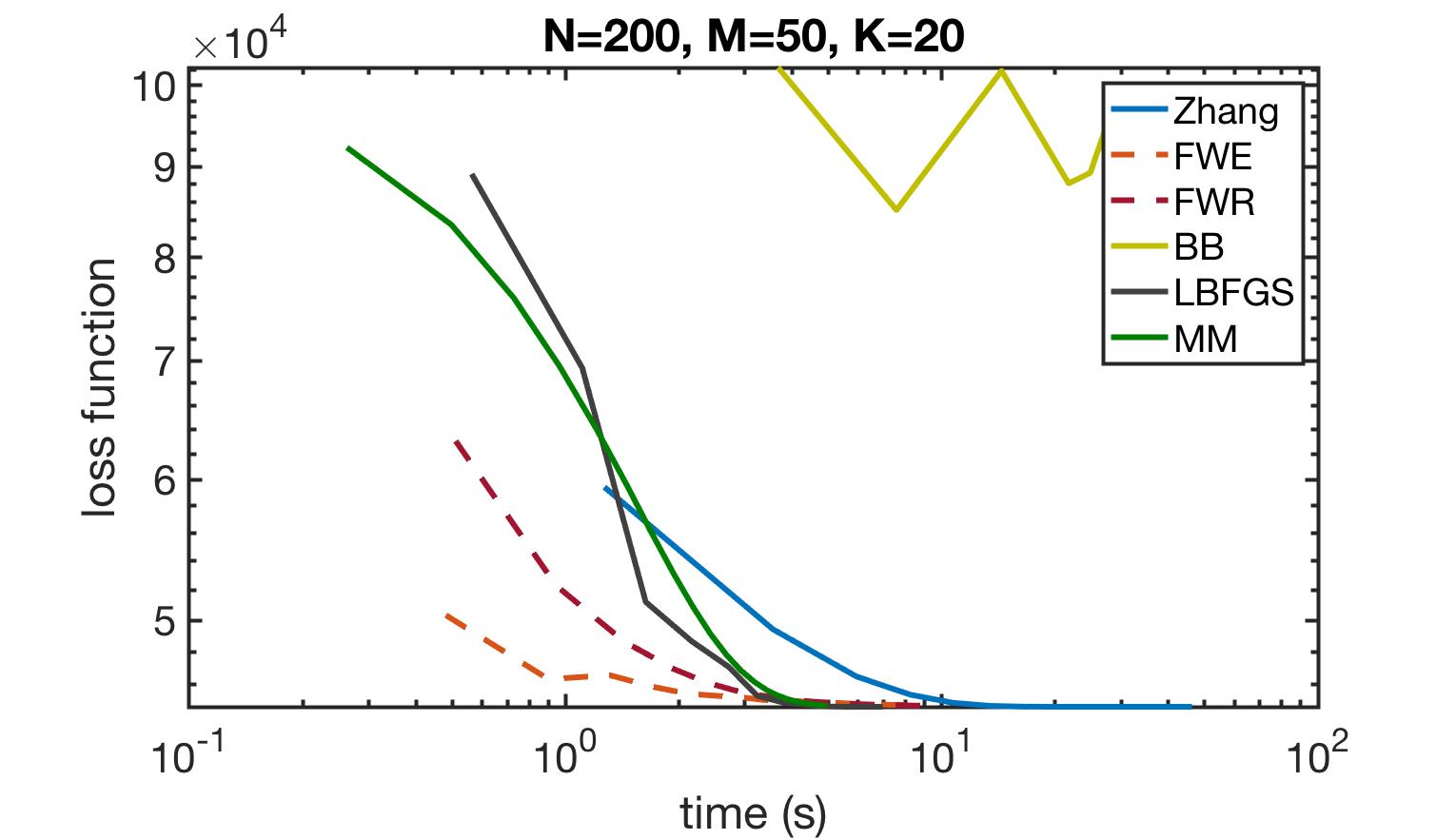}}
     \subfloat[$x_0=\frac{1}{2}(AM + HM)$]{\includegraphics[width=0.35\textwidth]{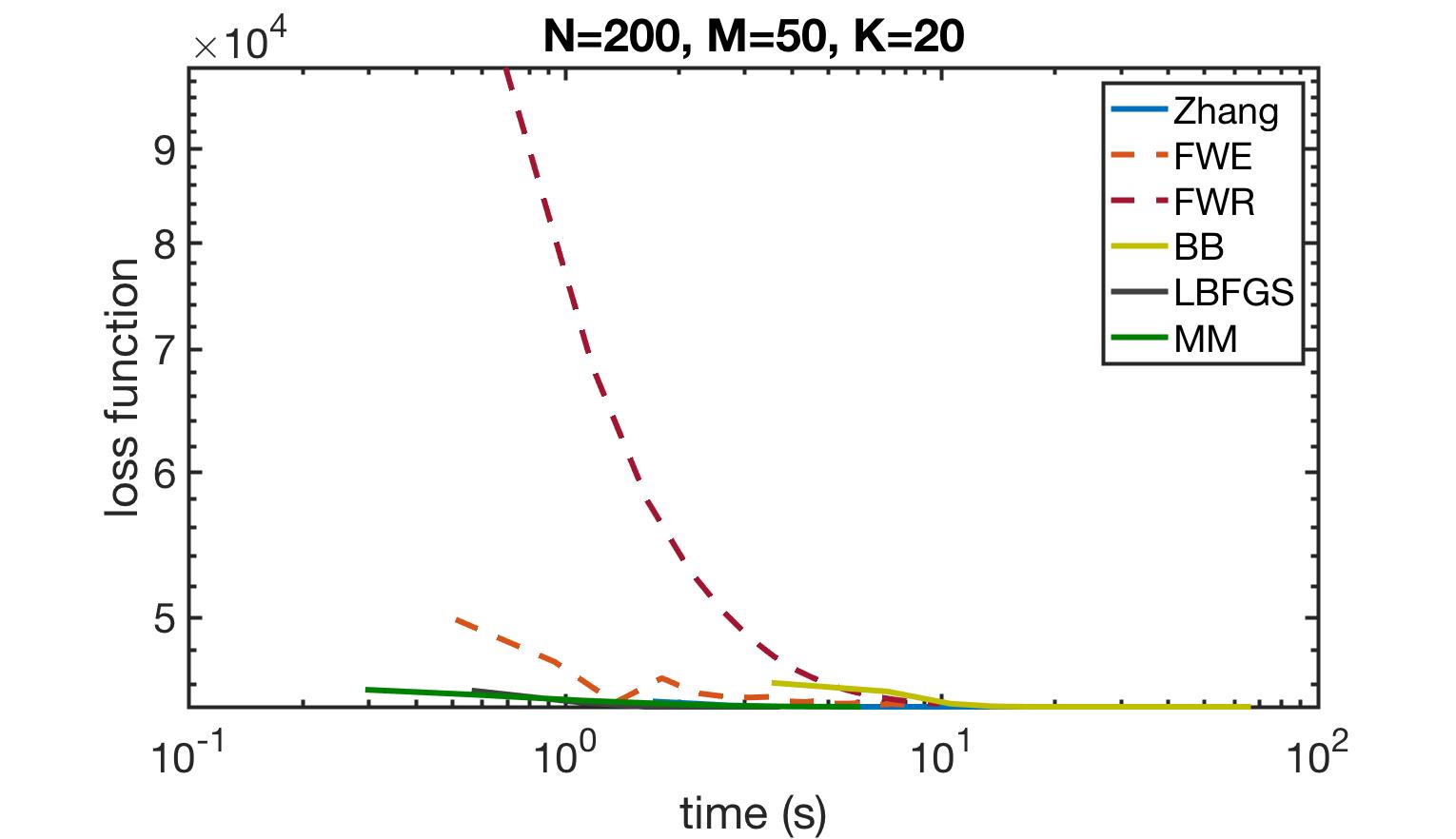}}
     \subfloat[$x_0=HM$]{\includegraphics[width=0.35\textwidth]{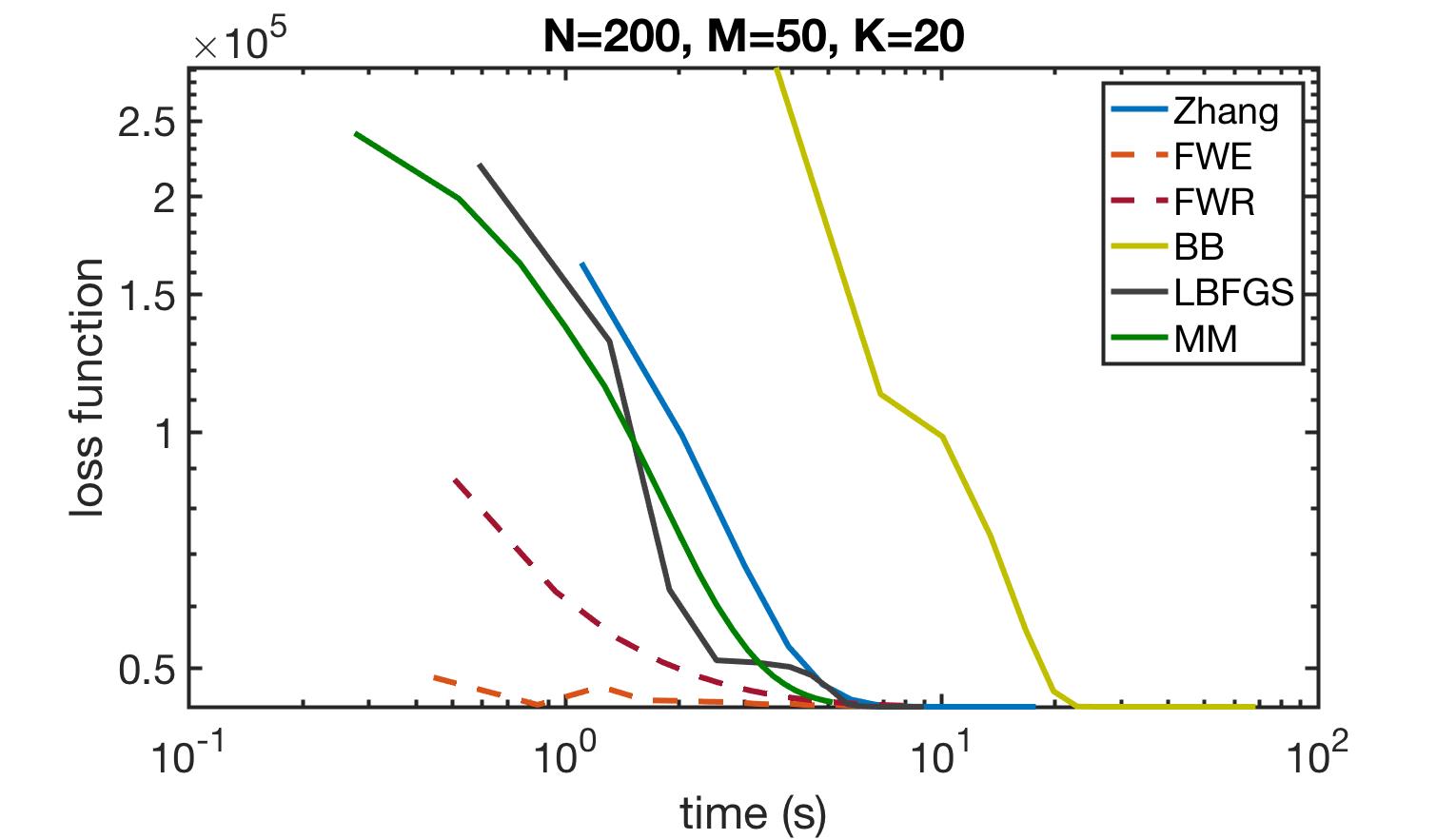}}
 \caption{Performance of \efw and \rfw in comparison with state-of-the-art methods for well-conditioned inputs and different initializations.} 
	\label{fig:init-w}
\end{figure}
\begin{figure}[!htbp]
     \subfloat[$x_0=A_1$]{\includegraphics[width=0.35\textwidth]{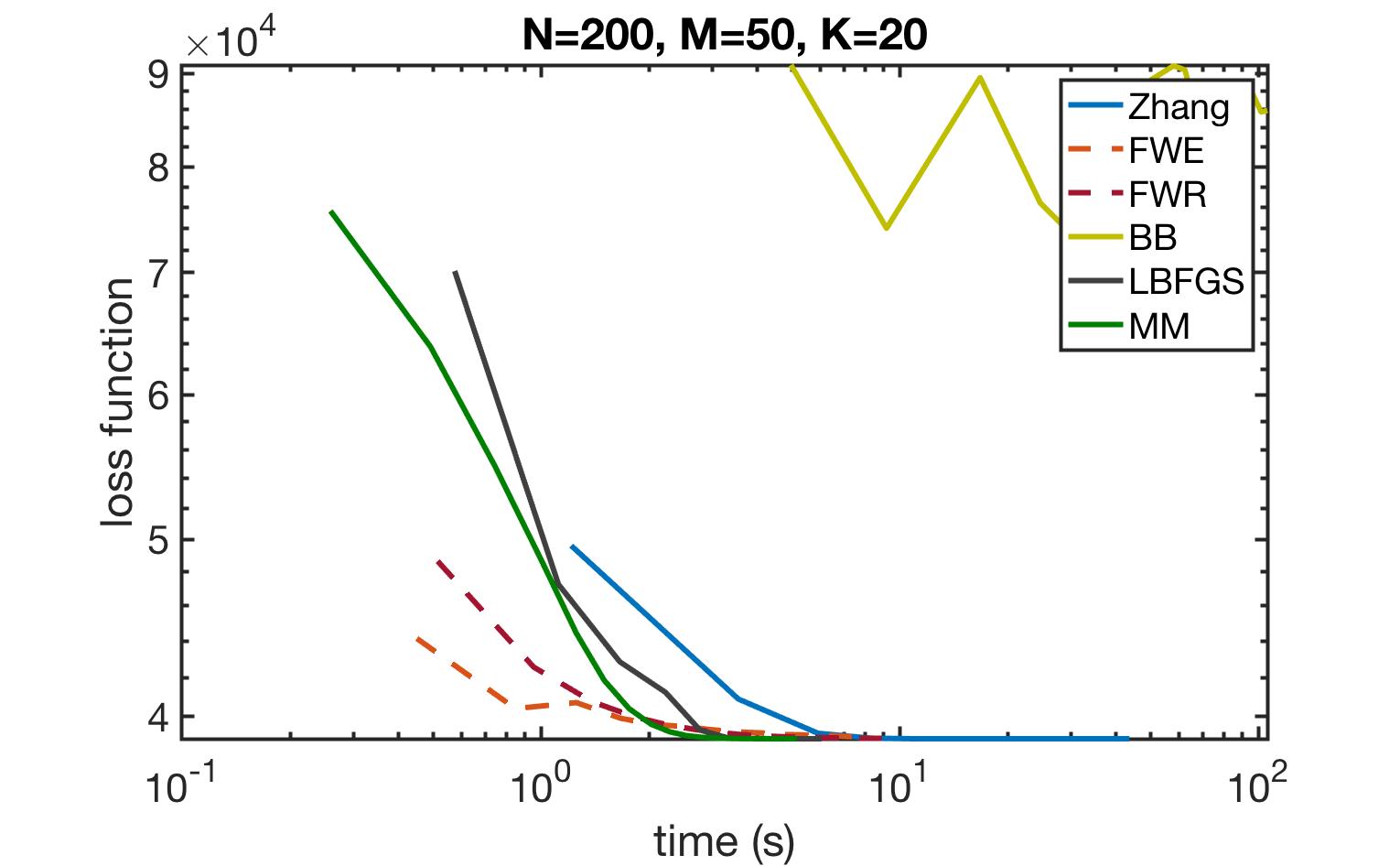}}
     \subfloat[$x_0=\frac{1}{2}(AM + HM)$]{\includegraphics[width=0.37\textwidth]{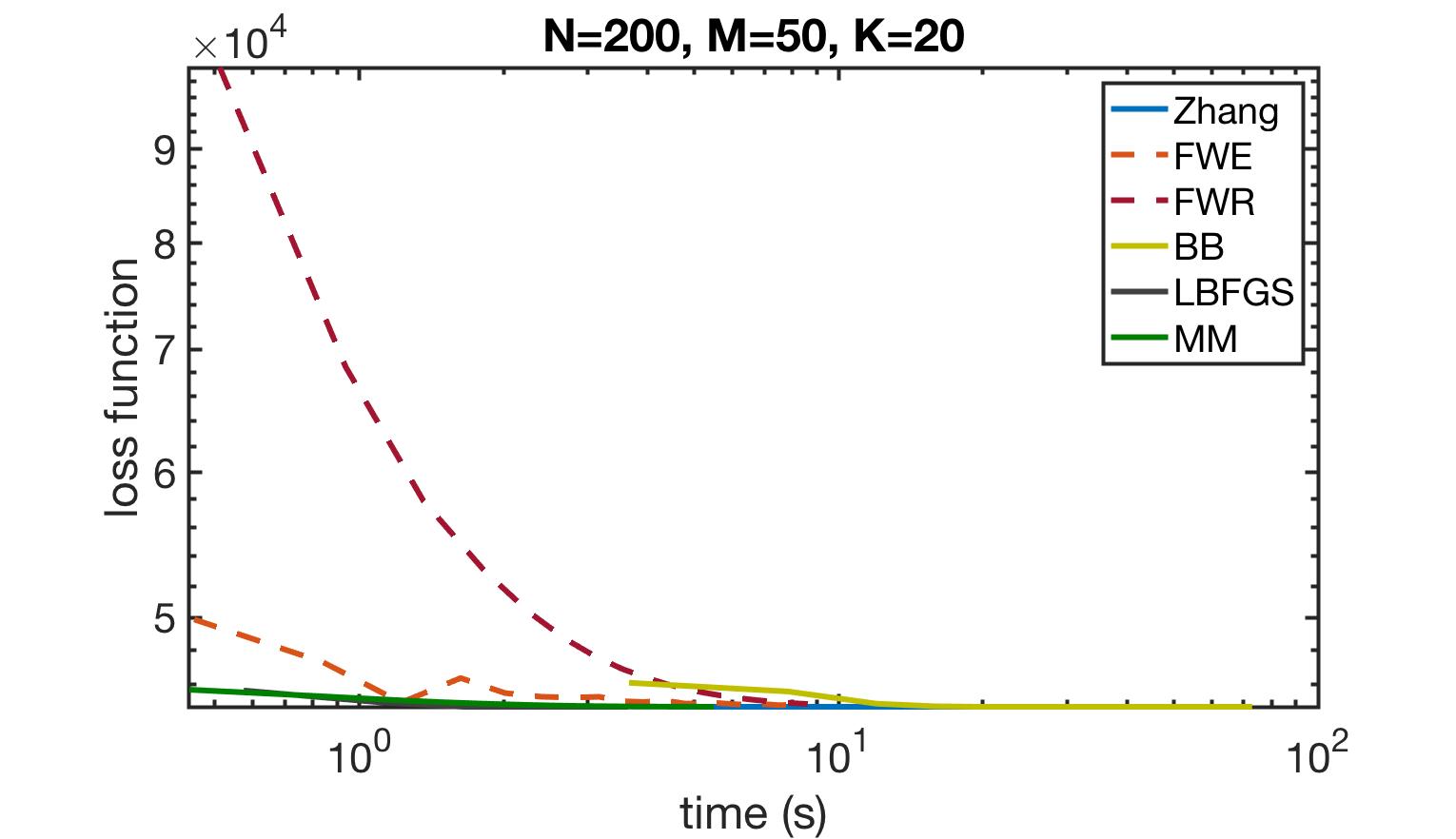}}
     \subfloat[$x_0=HM$]{\includegraphics[width=0.35\textwidth]{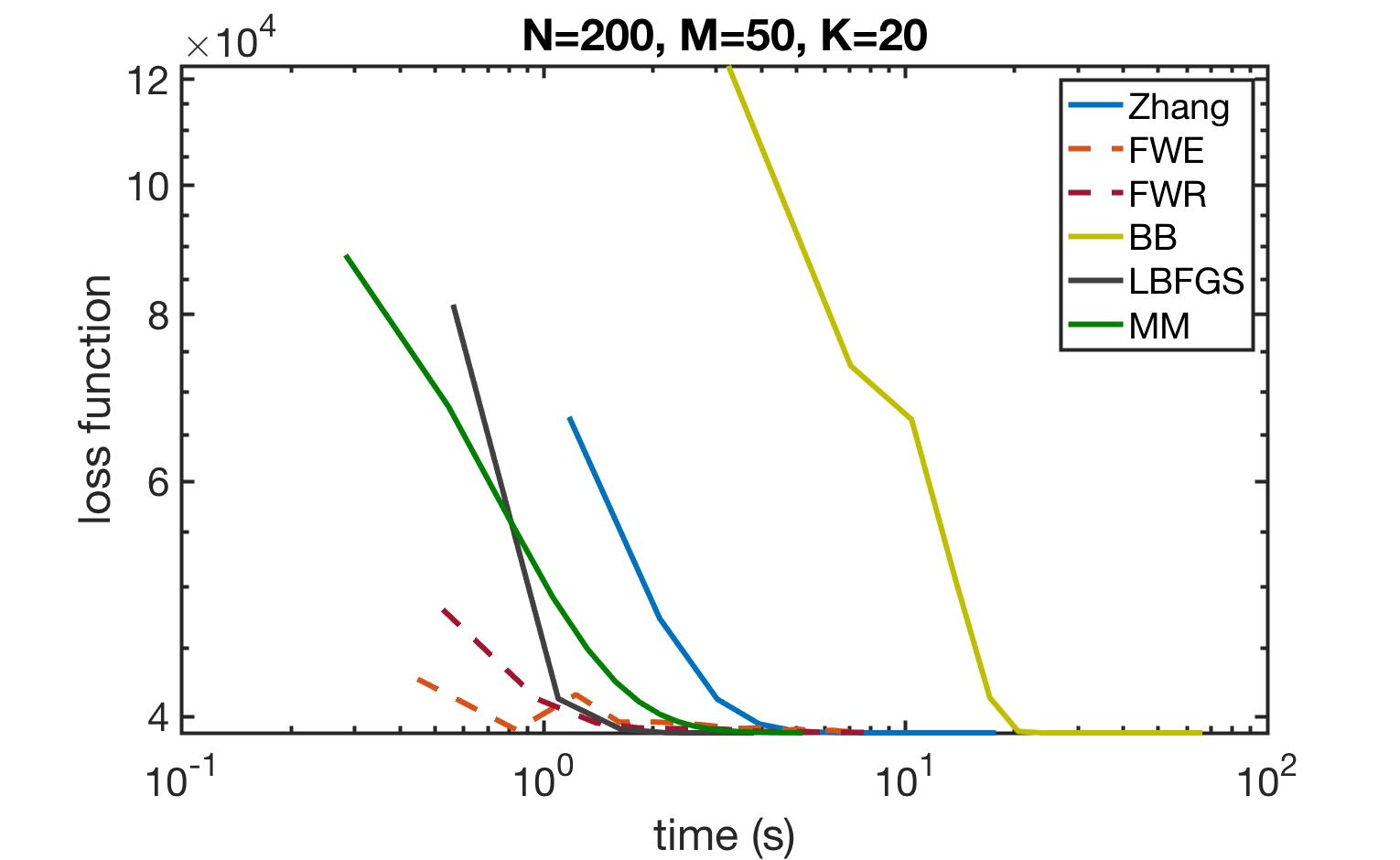}}
\caption{Performance of \efw and \rfw in comparison with state-of-the-art methods for ill-conditioned inputs and different initializations.} 
	\label{fig:init-i}
\end{figure}
\begin{figure}[!htbp]
     \includegraphics[width=\textwidth]{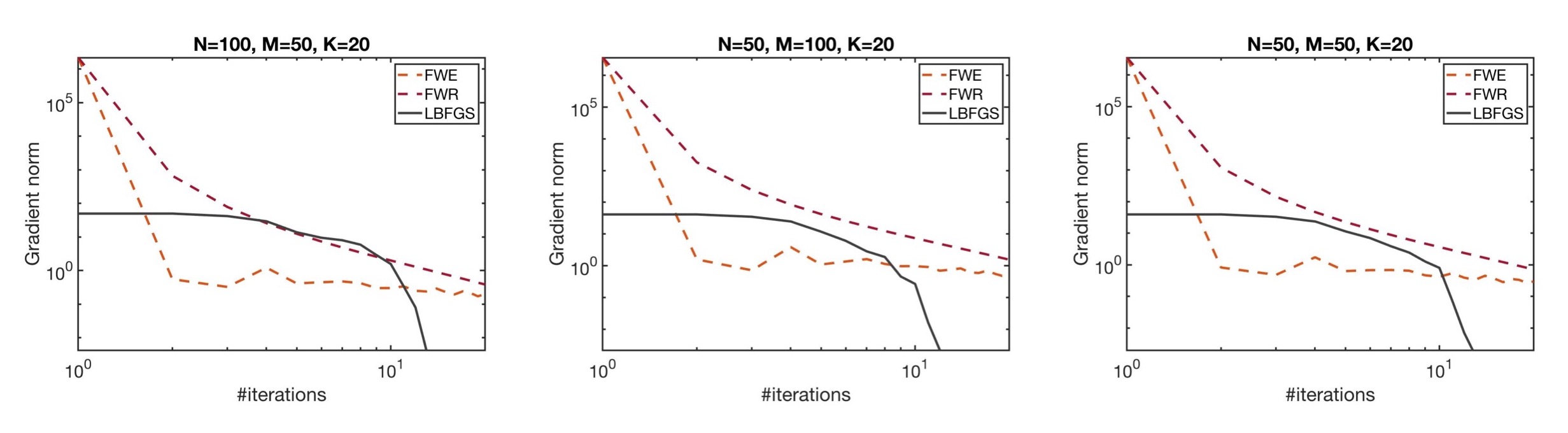}
\caption{Gradient norms at each iteration for \efw and \rfw in comparison with R-LBFGS.} 
\label{fig:gnorm}
\end{figure}

\subsubsection{Computing Wasserstein Barycenters}
To demonstrate the versatility of our approach, we use \rfw to compute Wasserstein barycenters. For this, we adapt the above described setup to implement Algorithm~\ref{alg.means}. Fig.~\ref{fig:wm} shows the performance of \rfw on both well- and ill-conditioned matrices. We compare three different initializations in each experiment: In (1), we choose an arbitrary element of $A$ as starting point ($X_0 \sim \mathcal{U}(A)$). The other experiments start at the lower and upper bounds of the constraint set, i.e., (2) $X_0$ is set to the arithmetic mean of $A$ and (3) $X_0=\alpha I$, where $\alpha$ is the smallest eigenvalue over $A$. Our results suggest that \rfw performs well when initialized from any point in the feasible region, even on sets of ill-conditioned matrices.
\begin{figure}[!h]
\centering
     \subfloat[Well-conditioned]{\includegraphics[width=5cm]{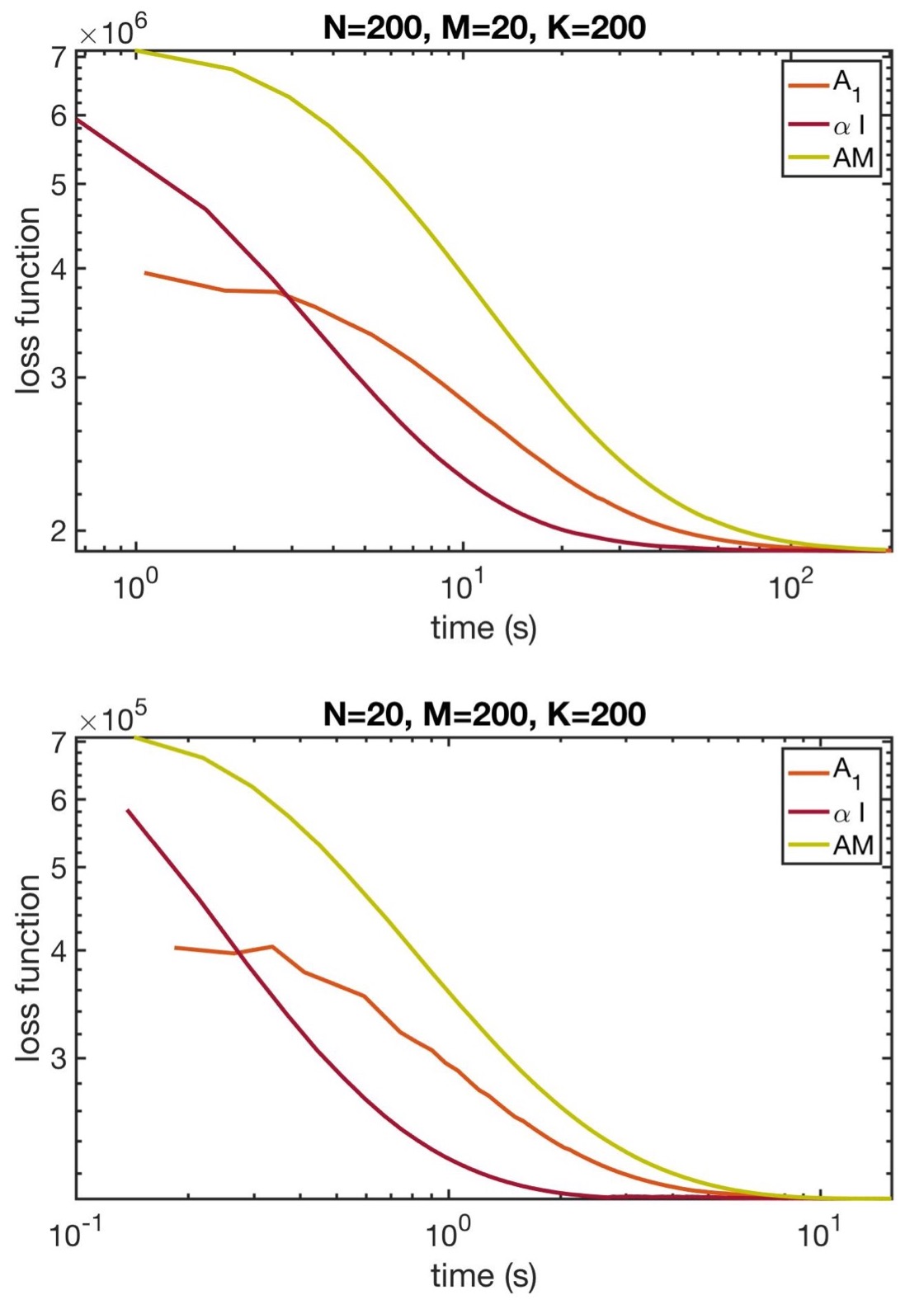}}
     \subfloat[Ill-conditioned]{\includegraphics[width=5.2cm]{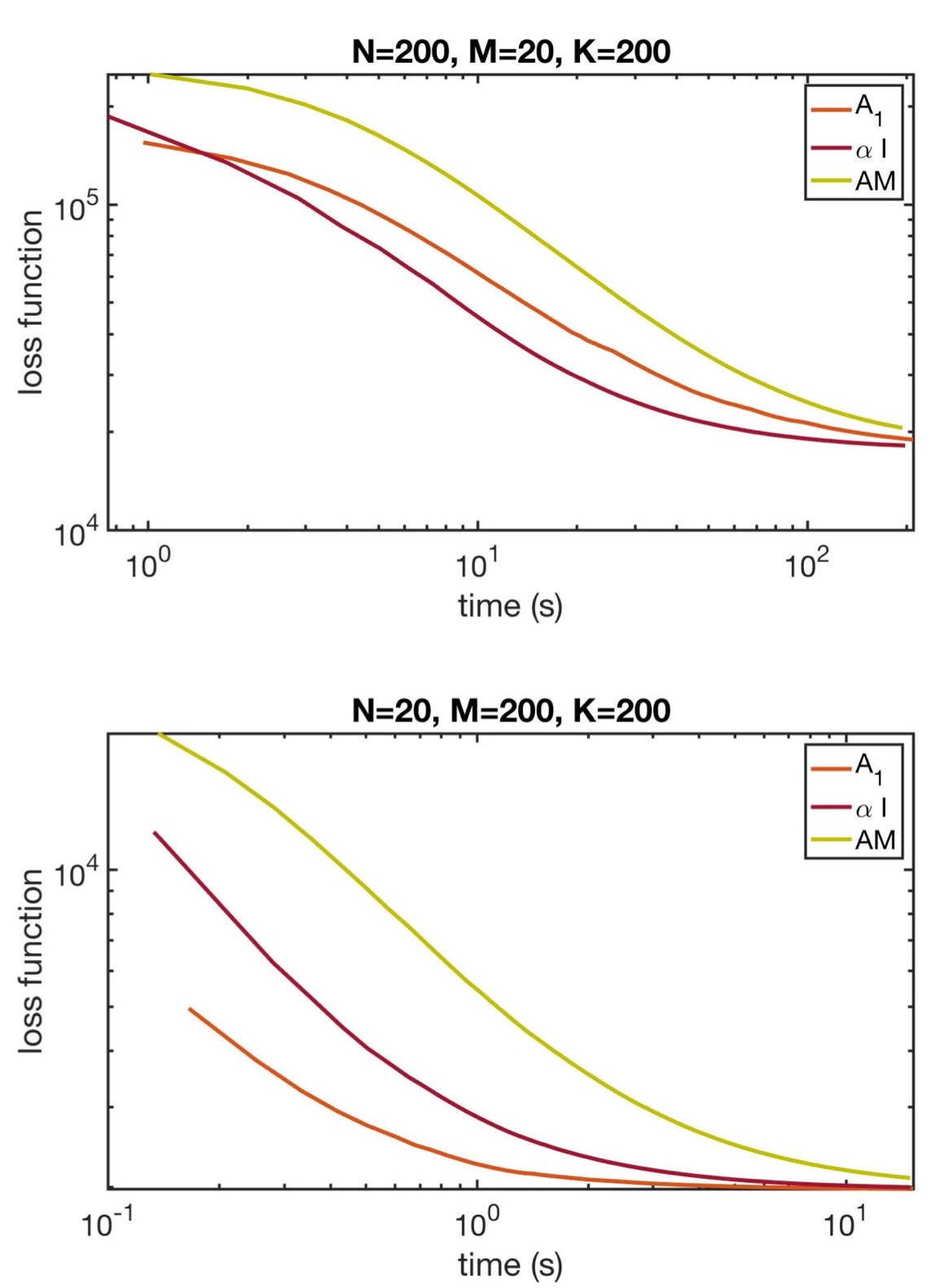}}
\caption{Performance \rfw for computing the Wasserstein mean for $M$ well- and $M$ ill-conditioned matrices of size $N$ for $K$ iteration and three initializations: (1) $X_0$ is set to one of the matrices in the set $A$, i.e., $A_i \in A$; (2) $X_0$ is initialized as the arithmetic mean of $A$, i.e., the upper bound of the constraint set; and (3) $X_0$ is set to the lower bound $\alpha I$ of the constraint, where $\alpha$ is the minimal eigenvalue over $A$.}
	\label{fig:wm}
\end{figure}
%
%\subsubsection{Learning DPP kernels}
%For both barycenter applications discussed above we solve the Riemannian linear oracle with respect to interval constraints. To demonstrate that our algorithm applies to other sets of constraints as well. Thm.~\ref{thm:cond} develops a closed form solution for the oracle under a condition number constraint. A notable example with such constraints is learning DPP kernels. We adapt the \rfw framework for fitting DPP kernels to sets $\mathcal{A} = \left(	A_1, \dots, A_n	\right) \subseteq \mathbb{R}^{d \times d}$, where we consider both well- and ill-conditioned matrices. All experiments are initialized to an arbitrary element of $\mathcal{A}$. In addition to the deterministic \rfw, we show results for the stochastic variants discussed above. Here, we observe significant performance gains.
%%
%\begin{figure}[!htbp]
%\centering
%     \includegraphics[width=0.7\textwidth]{dpp_ill_cond.png}
%\caption{Fitting a DPP kernel to ill-conditioned matrices.} 
%\label{fig:gnorm}
%\end{figure}

\subsubsection{Procrustes problem}
We implement \rfw for the special orthogonal group, focusing on the \emph{Procrustes problem} (see Algorithm~\ref{alg.procrustes}). Fig.~\ref{fig:procrustes} shows the performance of \rfw for different input sizes and initialization. 
The loss function is chosen as the optimization gap, where we compare with the analytically known optimal solution
\begin{align*}
X^* = U V^T \; ,
\end{align*}
where $AB^T = U \Sigma V^T$ is a singular value decomposition.

For initialization, we compare a random rotation matrix (i.e., $X_0 \sim \mathcal{U}(SO(n))$) with the identity matrix (i.e., $X_0=I_n$ ). In our experiments, initialization with the identity achieved superior performance. 
\begin{figure}[!h]
     \includegraphics[width=\textwidth]{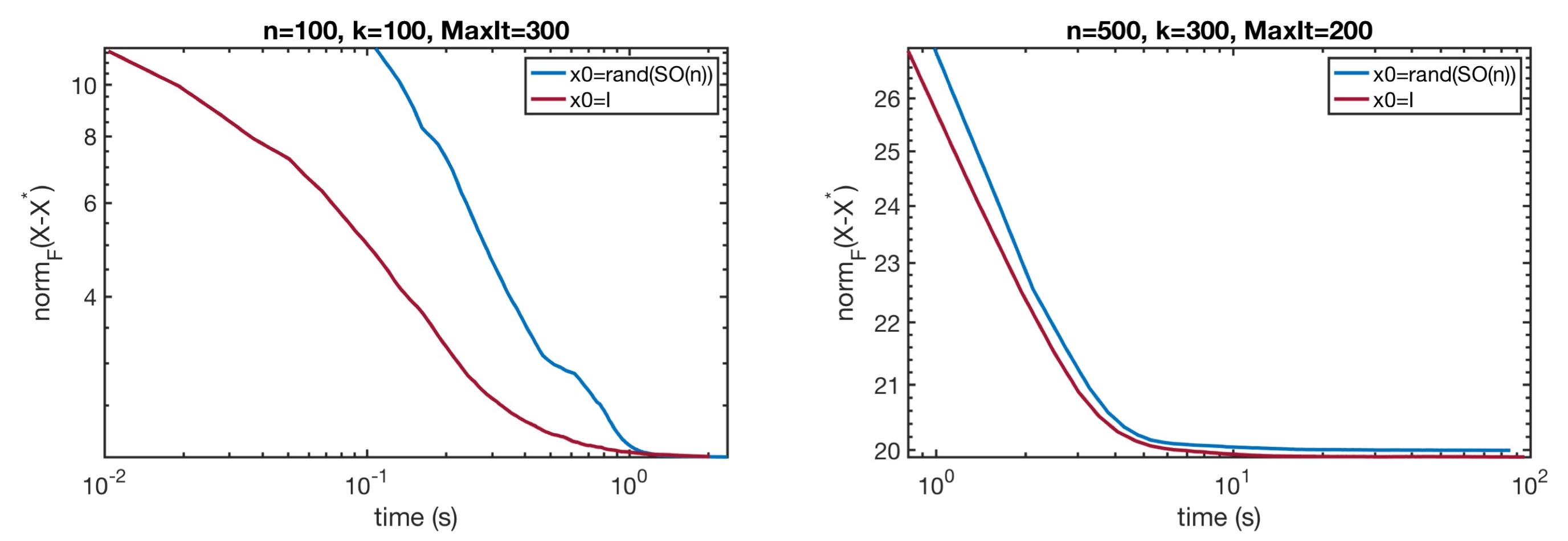}
\caption{Performance of \rfw for computing a solution of the procrustes problem. We consider two initializations, (i) $X_0 \sim \mathcal{U}(SO(n))$ and (ii) $X_0=I_n$ for input matrices $A, B \in \mathbb{R}^{n \times k}$. \emph{MaxIt} denotes the number of iterations.} 
\label{fig:procrustes}
\end{figure}

%%%%%%%%%%%%%%%%%%%%
\section{Discussion}
We presented a Riemannian version of the classic Frank-Wolfe method that enables constraint optimization on Riemannian (more precisely, on Hadamard) manifolds. Similar to the Euclidean case, we recover sublinear convergence rates for Riemannian Frank-Wolfe for both geodesically convex and nonconvex objectives. Under the stricter assumption of $\mu$-strongly g-convex objectives and a strict interiority condition on the constraint set, we show that even linear convergence rates can be attained by Riemannian Frank-Wolfe. To our knowledge, this work represents the first extension of Frank-Wolfe methods to a manifold setting. 

In addition to the general results, we present an efficient algorithm for optimization on Hermitian positive definite matrices. The key highlight of this specialization is a closed-form solution to the Riemannian ``linear'' oracle needed by Frank-Wolfe (this oracle involves solving a non-convex problem). While we focus on the specific problem of computing the Karcher mean (also known as the Riemannian centroid or geometric matrix mean), the derived closed form solutions apply to more general objective functions, and should be of wider interest for related non-convex and g-convex problems. To demonstrate this versatility, we also included an application of \rfw to the computation of Wasserstein barycenters on the Gaussian density manifold. Furthermore, we show that the Riemannian ``linear'' oracle can be solved in closed form also for the special orthogonal group. In future work, we hope to explore other constraint sets and matrix manifolds that admit an efficient solution to the Riemannian ``linear'' oracle.

Our algorithm is shown to be very competitive against a variety of established and recently proposed approaches~\cite{zhang2017, BFGS,BB} providing evidence for its applicability to large-scale statistics and machine learning problems. In follow up work~\citep{srfw}, we show that \rfw extends to nonconvex stochastic settings, further increasing the efficiency and versatility of Riemannian Frank-Wolfe methods. Exploring the use of this class of algorithms in large-scale machine learning applications is a promising avenue for future research.

%%%%%%%%%%%%%%%%%%%%%%%%%%%%%%%%
%%%%%%%%%%%%%%%%%%%%%%%%%%%%%%%%

\section*{Acknowledgements}
SS acknowledges support from NSF-IIS-1409802. This work was done during a visit of MW at MIT supported by a Dean’s grant from the Princeton University Graduate School.
The authors thank Charles Fefferman for helpful comments on the manuscript.

\bibliographystyle{plainnat}
\bibliography{quellen}

\begin{thebibliography}{67}
\providecommand{\natexlab}[1]{#1}
\providecommand{\url}[1]{\texttt{#1}}
\expandafter\ifx\csname urlstyle\endcsname\relax
  \providecommand{\doi}[1]{doi: #1}\else
  \providecommand{\doi}{doi: \begingroup \urlstyle{rm}\Url}\fi

\bibitem[Absil et~al.(2009)Absil, Mahony, and Sepulchre]{absil2009optimization}
P-A Absil, Robert Mahony, and Rodolphe Sepulchre.
\newblock \emph{Optimization algorithms on matrix manifolds}.
\newblock Princeton University Press, 2009.

\bibitem[Bach(2015)]{bach2015duality}
Francis Bach.
\newblock Duality between subgradient and conditional gradient methods.
\newblock \emph{SIAM Journal on Optimization}, 25\penalty0 (1):\penalty0
  115--129, 2015.

\bibitem[Bento et~al.(2017)Bento, Ferreira, and Melo]{bento2017iteration}
Glaydston~C Bento, Orizon~P Ferreira, and Jefferson~G Melo.
\newblock Iteration-complexity of gradient, subgradient and proximal point
  methods on {R}iemannian manifolds.
\newblock \emph{Journal of Optimization Theory and Applications}, 173\penalty0
  (2):\penalty0 548--562, 2017.

\bibitem[Bhatia(1997)]{bhatia97}
R.~Bhatia.
\newblock \emph{{Matrix Analysis}}.
\newblock Springer, 1997.

\bibitem[Bhatia(2007)]{bhatia07}
R.~Bhatia.
\newblock \emph{{Positive Definite Matrices}}.
\newblock Princeton University Press, 2007.

\bibitem[Bhatia and Holbrook(2006)]{BhatiaHolbrook}
R.~Bhatia and J.~Holbrook.
\newblock {R}iemannian geometry and matrix geometric means.
\newblock \emph{Linear Algebra Appl.}, 413:\penalty0 594–618, 2006.

\bibitem[Bhatia et~al.(2018{\natexlab{a}})Bhatia, Jain, and Lim]{bhatia}
Rajendra Bhatia, Tanvi Jain, and Yongdo Lim.
\newblock On the bures-wasserstein distance between positive definite matrices.
\newblock \emph{Expositiones Mathematicae}, 2018{\natexlab{a}}.
\newblock ISSN 0723-0869.
\newblock \doi{https://doi.org/10.1016/j.exmath.2018.01.002}.

\bibitem[Bhatia et~al.(2018{\natexlab{b}})Bhatia, Jain, and Lim]{bhatia2}
Rajendra Bhatia, Tanvi Jain, and Yongdo Lim.
\newblock Strong convexity of sandwiched entropies and related optimization
  problems.
\newblock \emph{Reviews in Mathematical Physics}, 30\penalty0 (09):\penalty0
  1850014, 2018{\natexlab{b}}.

\bibitem[Bini and Iannazzo(2013)]{mmtoolbox}
D.~A. Bini and B.~Iannazzo.
\newblock Computing the {K}archer mean of symmetric positive definite matrices.
\newblock \emph{Linear Algebra and its Applications}, 438\penalty0
  (4):\penalty0 1700--10, Oct. 2013.

\bibitem[Boumal et~al.(2014)Boumal, Mishra, Absil, and Sepulchre]{manopt}
N.~Boumal, B.~Mishra, P.-A. Absil, and R.~Sepulchre.
\newblock {M}anopt, a {M}atlab toolbox for optimization on manifolds.
\newblock \emph{Journal of Machine Learning Research}, 15:\penalty0 1455--1459,
  2014.
\newblock URL \url{http://www.manopt.org}.

\bibitem[Boumal et~al.(2016)Boumal, Absil, and Cartis]{boumal2016global}
Nicolas Boumal, P-A Absil, and Coralia Cartis.
\newblock Global rates of convergence for nonconvex optimization on manifolds.
\newblock \emph{arXiv preprint arXiv:1605.08101}, 2016.

\bibitem[Calinescu et~al.(2011)Calinescu, Chekuri, P\'{a}l, and
  Vondr\'{a}k]{caliChVon11}
G.~Calinescu, C.~Chekuri, M.~P\'{a}l, and J.~Vondr\'{a}k.
\newblock Maximizing a submodular set function subject to a matroid constraint.
\newblock \emph{SIAM J. Computing}, 40\penalty0 (6), 2011.

\bibitem[Canon and Cullum(1968)]{canon-cullum}
M.~Canon and C.~Cullum.
\newblock A tight upper bound on the rate of convergence of frank-wolfe
  algorithm.
\newblock \emph{Siam Journal on Control}, 6:\penalty0 509--516, 1968.

\bibitem[Carson et~al.(2017)Carson, Mixon, and Villar]{villar}
Timothy Carson, Dustin~G. Mixon, and Soledad Villar.
\newblock Manifold optimization for k-means clustering.
\newblock In \emph{2017 International Conference on Sampling Theory and
  Applications (SampTA)}, pages 73--77, 2017.
\newblock \doi{10.1109/SAMPTA.2017.8024388}.

\bibitem[Chavel(2006)]{chavel2006riemannian}
Isaac Chavel.
\newblock \emph{{R}iemannian {G}eometry: {A} modern introduction}, volume~98.
\newblock Cambridge university press, 2006.

\bibitem[Cherian and Sra(2015)]{cherian2015riemannian}
Anoop Cherian and Suvrit Sra.
\newblock Riemannian dictionary learning and sparse coding for positive
  definite matrices.
\newblock \emph{arXiv:1507.02772}, 2015.

\bibitem[Clarkson(2010)]{clarkson}
Kenneth~L. Clarkson.
\newblock Coresets, sparse greedy approximation, and the frank-wolfe algorithm.
\newblock \emph{ACM Trans. Algorithms}, 6\penalty0 (4), September 2010.
\newblock ISSN 1549-6325.
\newblock URL \url{https://doi.org/10.1145/1824777.1824783}.

\bibitem[Combettes and Pokutta(2021)]{COMBETTES2021565}
Cyrille~W. Combettes and Sebastian Pokutta.
\newblock Complexity of linear minimization and projection on some sets.
\newblock \emph{Operations Research Letters}, 49\penalty0 (4):\penalty0
  565--571, 2021.

\bibitem[Edelman et~al.(1998)Edelman, Arias, and Smith]{edelman98}
A.~Edelman, T.~A. Arias, and S.~T. Smith.
\newblock The geometry of algorithms with orthogonality constraints.
\newblock \emph{SIAM J. Matrix Analysis and Applications (SIMAX)}, 20\penalty0
  (2):\penalty0 303--353, 1998.

\bibitem[Frank and Wolfe(1956)]{fw_original}
M.~Frank and P.~Wolfe.
\newblock An algorithm for quadratic programming.
\newblock \emph{Naval Research Logistics Quarterly}, 3(95, 1956.

\bibitem[Fujishige and Isotani(2011)]{fujiso11}
S.~Fujishige and S.~Isotani.
\newblock A submodular function minimization algorithm based on the
  minimum-norm base.
\newblock \emph{Pacific Journal of Optimization}, 7:\penalty0 3--17, 2011.

\bibitem[Garber and Hazan(2015)]{GH15}
D.~Garber and E.~Hazan.
\newblock Faster rates for the {F}rank-{W}olfe method over strongly-convex
  sets.
\newblock In \emph{International Conference on Machine Learning}, pages
  541--549, 2015.

\bibitem[Gu{\'e}Lat and Marcotte(1986)]{marcotte}
J.~Gu{\'e}Lat and P.~Marcotte.
\newblock Some comments on {W}olfe's `away step'.
\newblock \emph{Mathematical Programming}, 35\penalty0 (1):\penalty0 110--119,
  May 1986.
\newblock ISSN 1436-4646.
\newblock \doi{10.1007/BF01589445}.
\newblock URL \url{https://doi.org/10.1007/BF01589445}.

\bibitem[Hazan and Luo(2016)]{hazan2016variance}
Elad Hazan and Haipeng Luo.
\newblock Variance-reduced and projection-free stochastic optimization.
\newblock In \emph{International Conference on Machine Learning}, pages
  1263--1271, 2016.

\bibitem[Helmke et~al.(2007)Helmke, H{\"u}per, Lee, and
  Moore]{helmke2007essential}
Uwe Helmke, Knut H{\"u}per, Pei~Yean Lee, and John Moore.
\newblock Essential matrix estimation using {G}auss-{N}ewton iterations on a
  manifold.
\newblock \emph{International Journal of Computer Vision}, 74\penalty0
  (2):\penalty0 117--136, 2007.

\bibitem[Holloway(1974)]{holloway}
Charles~A. Holloway.
\newblock An extension of the frank and wolfe method of feasible directions.
\newblock \emph{Mathematical Programming}, 6:\penalty0 14–27, 1974.

\bibitem[Hosseini and Sra(2015)]{hoSr15b}
Reshad Hosseini and Suvrit Sra.
\newblock Matrix manifold optimization for {G}aussian mixtures.
\newblock In \emph{NIPS}, 2015.

\bibitem[Iannazzo and Porcelli(2018)]{BB}
Bruno Iannazzo and Margherita Porcelli.
\newblock The riemannian barzilai–borwein method with nonmonotone line search
  and the matrix geometric mean computation.
\newblock \emph{IMA Journal of Numerical Analysis}, 38\penalty0 (1):\penalty0
  495--517, 2018.
\newblock \doi{10.1093/imanum/drx015}.
\newblock URL \url{http://dx.doi.org/10.1093/imanum/drx015}.

\bibitem[Jaggi(2013)]{jaggi2013revisiting}
Martin Jaggi.
\newblock Revisiting {F}rank-{W}olfe: {P}rojection-free sparse convex
  optimization.
\newblock In \emph{International Conference on Machine Learning (ICML)}, pages
  427--435, 2013.

\bibitem[Jeuris et~al.(2012)Jeuris, Vandebril, and Vandereycken]{gm_survey}
B.~Jeuris, R.~Vandebril, and B.~Vandereycken.
\newblock A survey and comparison of contemporary algorithms for computing the
  matrix geometric mean.
\newblock \emph{Electronic Transactions on Numerical Analysis}, 39:\penalty0
  379--402, 2012.

\bibitem[Jost(2011)]{jost}
J.~Jost.
\newblock \emph{{Riemannian Geometry and Geometric Analysis}}.
\newblock Springer, 2011.

\bibitem[Karcher(1977)]{karcher}
H.~Karcher.
\newblock {R}iemannian center of mass and mollifier smoothing.
\newblock \emph{Comm. Pure Appl. Math.}, 30(5):\penalty0 509--541, 1977.

\bibitem[Karimi et~al.(2016)Karimi, Nutini, and Schmidt]{KNS16}
H.~Karimi, J.~Nutini, and M.~W. Schmidt.
\newblock Linear convergence of gradient and proximal-gradient methods under
  the {P}olyak-{L}ojasiewicz condition.
\newblock \emph{CoRR}, abs/1608.04636, 2016.

\bibitem[Kubo and Ando(1979)]{KuboAndo}
F.~Kubo and T.~Ando.
\newblock Means of positive linear operators.
\newblock \emph{Math. Ann.}, 246:\penalty0 205–224, 1979.

\bibitem[Lacoste-Julien(2016)]{LJ16}
S.~Lacoste-Julien.
\newblock Convergence rate of {F}rank-{W}olfe for non-convex objectives.
\newblock \emph{arXiv preprint arXiv:1607.00345}, 2016.

\bibitem[Lacoste-Julien and Jaggi(2015)]{julien15}
Simon Lacoste-Julien and Martin Jaggi.
\newblock On the global linear convergence of {F}rank-{W}olfe optimization
  variants.
\newblock In \emph{Proceedings of the 28th International Conference on Neural
  Information Processing Systems - Volume 1}, NIPS'15, pages 496--504,
  Cambridge, MA, USA, 2015. MIT Press.
\newblock URL \url{http://dl.acm.org/citation.cfm?id=2969239.2969295}.

\bibitem[Lawson and Lim(2014)]{LL14}
J.~Lawson and Y.~Lim.
\newblock {K}archer means and {K}archer equations of positive definite
  operators.
\newblock \emph{Trans. Amer. Math. Soc. Ser. B}, 1:\penalty0 1--22, 2014.

\bibitem[Le~Bihan et~al.(2001)Le~Bihan, Mangin, Poupon, Clark, Pappata, Molko,
  and Chabriat]{le2001diffusion}
Denis Le~Bihan, Jean-Fran{\c{c}}ois Mangin, Cyril Poupon, Chris~A Clark, Sabina
  Pappata, Nicolas Molko, and Hughes Chabriat.
\newblock Diffusion {T}ensor {I}maging: {C}oncepts and {A}pplications.
\newblock \emph{Journal of magnetic resonance imaging}, 13\penalty0
  (4):\penalty0 534--546, 2001.

\bibitem[Ledyaev et~al.(2006)Ledyaev, Treiman, and Zhu]{ledyaev2006helly}
Yuri~S Ledyaev, Jay~S Treiman, and Qiji~J Zhu.
\newblock Helly's intersection theorem on manifolds of nonpositive curvature.
\newblock \emph{Journal of Convex Analysis}, 13\penalty0 (3/4):\penalty0 785,
  2006.

\bibitem[Lim and P{\'a}lfia(2012)]{lim.palfia}
Yongdo Lim and Mikl{\'o}s P{\'a}lfia.
\newblock Matrix power means and the {K}archer mean.
\newblock \emph{Journal of Functional Analysis}, 262\penalty0 (4):\penalty0
  1498--1514, 2012.

\bibitem[Liu and Boumal(2019)]{liu2019simple}
Changshuo Liu and Nicolas Boumal.
\newblock Simple algorithms for optimization on riemannian manifolds with
  constraints, 2019.

\bibitem[Lojasiewicz(1963)]{losj63}
S.~Lojasiewicz.
\newblock Une propri{\'e}t{\'e} topologique des sous-ensembles analytiques
  r{\'e}els.
\newblock \emph{Les {\'e}quations aux d{\'e}riv{\'e}es partielles},
  117:\penalty0 87--89, 1963.

\bibitem[Malag{\`o} et~al.(2018)Malag{\`o}, Montrucchio, and Pistone]{Malago}
Luigi Malag{\`o}, Luigi Montrucchio, and Giovanni Pistone.
\newblock Wasserstein riemannian geometry of positive-definite matrices ?
\newblock 2018.

\bibitem[Mariet and Sra(2015)]{Mariet2015FixedpointAF}
Zelda~E. Mariet and S.~Sra.
\newblock Fixed-point algorithms for learning determinantal point processes.
\newblock In \emph{ICML}, 2015.

\bibitem[Mitchell et~al.(1974)Mitchell, Dem’yanov, and Malozemov]{mitchell74}
B.~F. Mitchell, V.~F. Dem’yanov, and V.~N. Malozemov.
\newblock Finding the point of a polyhedron closest to the origin.
\newblock \emph{SIAM Journal on Control}, 12\penalty0 (1):\penalty0 19--26,
  1974.

\bibitem[Moakher(2002)]{moakher2002means}
Maher Moakher.
\newblock Means and averaging in the group of rotations.
\newblock \emph{SIAM journal on matrix analysis and applications}, 24\penalty0
  (1):\penalty0 1--16, 2002.

\bibitem[Montanari and Richard(2016)]{montanari}
Andrea Montanari and Emile Richard.
\newblock Non-negative principal component analysis: Message passing algorithms
  and sharp asymptotics.
\newblock \emph{IEEE Transactions on Information Theory}, 62\penalty0
  (3):\penalty0 1458--1484, 2016.
\newblock \doi{10.1109/TIT.2015.2457942}.

\bibitem[Nielsen and Bhatia(2013)]{nieBha13}
F.~Nielsen and R.~Bhatia, editors.
\newblock \emph{{Matrix Information Geometry}}.
\newblock Springer, 2013.

\bibitem[Polyak(1963)]{polyak63}
B.~T. Polyak.
\newblock Gradient methods for minimizing functionals (in {R}ussian).
\newblock \emph{Zh. Vychisl. Mat. Mat. Fiz.}, page 643–653, 1963.

\bibitem[Polyak(1987)]{polyak}
B.~T. Polyak.
\newblock \emph{Introduction to {O}ptimization}.
\newblock Optimization Software Inc., 1987.
\newblock Nov 2010 revision.

\bibitem[Pálfia(2016)]{PALFIA2016951}
Miklós Pálfia.
\newblock Operator means of probability measures and generalized karcher
  equations.
\newblock \emph{Advances in Mathematics}, 289:\penalty0 951--1007, 2016.
\newblock URL
  \url{https://www.sciencedirect.com/science/article/pii/S000187081500479X}.

\bibitem[Reddi et~al.(2016)Reddi, Sra, P{\'o}czos, and
  Smola]{reddi2016stochastic}
Sashank~J Reddi, Suvrit Sra, Barnab{\'a}s P{\'o}czos, and Alex Smola.
\newblock Stochastic {F}rank-{W}olfe methods for nonconvex optimization.
\newblock In \emph{Communication, Control, and Computing (Allerton), 2016 54th
  Annual Allerton Conference on}, pages 1244--1251. IEEE, 2016.

\bibitem[Ring and Wirth(2012)]{ring2012optimization}
Wolfgang Ring and Benedikt Wirth.
\newblock Optimization methods on {R}iemannian manifolds and their application
  to shape space.
\newblock \emph{SIAM Journal on Optimization}, 22\penalty0 (2):\penalty0
  596--627, 2012.

\bibitem[Sra and Hosseini(2015)]{SH15}
S.~Sra and R.~Hosseini.
\newblock Conic geometric optimization on the manifold of positive definite
  matrices.
\newblock \emph{SIAM Journal on Optimization}, 25\penalty0 (1):\penalty0
  713--739, 2015.

\bibitem[Sra and Hosseini(2013)]{sra2013geometric}
Suvrit Sra and Reshad Hosseini.
\newblock Geometric optimisation on positive definite matrices for elliptically
  contoured distributions.
\newblock In \emph{Advances in Neural Information Processing Systems}, pages
  2562--2570, 2013.

\bibitem[Sun et~al.(2015)Sun, Qu, and Wright]{sun2015complete}
Ju~Sun, Qing Qu, and John Wright.
\newblock Complete {D}ictionary {R}ecovery over the {S}phere {II}: {R}ecovery
  by {R}iemannian {T}rust-region {M}ethod.
\newblock \emph{arXiv:1511.04777}, 2015.

\bibitem[Tan et~al.(2014)Tan, Tsang, Wang, Vandereycken, and
  Pan]{tan2014riemannian}
Mingkui Tan, Ivor~W Tsang, Li~Wang, Bart Vandereycken, and Sinno~J Pan.
\newblock Riemannian pursuit for big matrix recovery.
\newblock In \emph{International Conference on Machine Learning (ICML-14)},
  pages 1539--1547, 2014.

\bibitem[Udriste(1994)]{udriste1994convex}
Constantin Udriste.
\newblock \emph{Convex functions and optimization methods on Riemannian
  manifolds}, volume 297.
\newblock Springer Science \& Business Media, 1994.

\bibitem[Vandereycken(2013)]{vandereycken2013low}
Bart Vandereycken.
\newblock Low-rank matrix completion by {R}iemannian optimization.
\newblock \emph{SIAM Journal on Optimization}, 23\penalty0 (2):\penalty0
  1214--1236, 2013.

\bibitem[Weber and Sra(2019)]{srfw}
Melanie Weber and Suvrit Sra.
\newblock Nonconvex stochastic optimization on manifolds via {R}iemannian
  {F}rank-{W}olfe methods.
\newblock \emph{arXiv:1910.04194}, 2019.

\bibitem[Wolfe(1970)]{wolfe}
P.~Wolfe.
\newblock Convergence theory in nonlinear programming.
\newblock \emph{Integer and Nonlinear Programming}, 1970.

\bibitem[Yuan et~al.(2016)Yuan, Huang, Absil, and Gallivan]{BFGS}
Xinru Yuan, Wen Huang, P.-A. Absil, and K.A. Gallivan.
\newblock A riemannian limited-memory bfgs algorithm for computing the matrix
  geometric mean.
\newblock \emph{Procedia Computer Science}, 80:\penalty0 2147 -- 2157, 2016.
\newblock ISSN 1877-0509.
\newblock \doi{https://doi.org/10.1016/j.procs.2016.05.534}.
\newblock International Conference on Computational Science 2016, ICCS 2016,
  6-8 June 2016, San Diego, California, USA.

\bibitem[Yuan et~al.(2017)Yuan, Huang, Absil, and Gallivan]{YHAG2017}
Xinru Yuan, Wen Huang, P.-A. Absil, and K.~A. Gallivan.
\newblock A {R}iemannian quasi-{N}ewton method for computing the {K}archer mean
  of symmetric positive definite matrices.
\newblock \emph{Florida State University}, \penalty0 (FSU17-02), 2017.

\bibitem[Zhang and Sra(2016)]{ZS16}
H.~Zhang and S.~Sra.
\newblock First-order methods for geodesically convex optimization.
\newblock In \emph{Conference on Learning Theory (COLT)}, Jun. 2016.

\bibitem[Zhang et~al.(2016)Zhang, Reddi, and Sra]{rsvrg}
H.~Zhang, S.~Reddi, and S.~Sra.
\newblock Fast stochastic optimization on {R}iemannian manifolds.
\newblock In \emph{Advances in Neural Information Processing Systems (NIPS)},
  2016.

\bibitem[Zhang(2017)]{zhang2017}
T.~Zhang.
\newblock A majorization-minimization algorithm for computing the {K}archer
  mean of positive definite matrices.
\newblock \emph{SIAM Journal on Matrix Analysis and Applications}, 38\penalty0
  (2):\penalty0 387--400, 2017.
\newblock \doi{10.1137/15M1024482}.

\bibitem[Zhang et~al.(2013)Zhang, Wiesel, and Greco]{zhang2013multivariate}
Teng Zhang, Ami Wiesel, and Maria~S Greco.
\newblock Multivariate generalized {G}aussian distribution: {C}onvexity and
  graphical models.
\newblock \emph{Signal Processing, IEEE Transactions on}, 61\penalty0
  (16):\penalty0 4141--4148, 2013.

\end{thebibliography}

\appendix

\section{Omitted proofs}\label{app:A}

\subsection{Convergence of nonconvex RFW}
We recall the convergence result for \rfw in the nonconvex setting and then provide the proof that was omitted in the main text.
\begin{theorem}[Theorem~\ref{thm:conv-FWR-nonconvex}]
  Let $\tilde{G}_k := \min_{0 \leq k \leq K} G(X_k)$ (where $G(X_k)$ denotes the Frank-Wolfe gap at $X_k$). After $K$ iterations of Algorithm~\ref{alg.rfw}, we have $\tilde{G}_k \leq \frac{\max \lbrace	2 h_0, M_{\phi}	\rbrace}{\sqrt{K+1}}$.
\end{theorem}
\begin{proof}
We follow the proof of Theorem~\ref{thm:conv-FWR} to get a bound on the update and then substitute the FW gap in the directional term:
\begin{align}
\phi(X_{k+1}) &\leq \phi(X_k) + s_k \underbrace{\ip{\grad \phi(X_k)}{\Exp_{X_k}^{-1}(Z_k)}}_{=-G(X_k)} + \frac{s_k ^2 M_{\phi}}{2} \\
&\leq \phi(X_k)  - s_k G(X_k) + \frac{s_k ^2 M_{\phi}}{2} \; .
\end{align}
The step size that minimizes the $s_k$-terms on the right hand side is given by
\begin{equation}
s_k ^{*} = \min \bigg \lbrace \frac{G(X_k)}{M_{\phi}}, 1 \bigg \rbrace \; ,
\end{equation}
which gives
\begin{equation}
\phi(X_{k+1}) \leq \phi(X_k) - \min \bigg \lbrace \frac{G(X_k)^2}{2 M_{\phi}}, G(X_k) - \frac{M_{\phi}}{2} \cdot \mathbbm{1}_{\lbrace G(X_k) > M_{\phi}	\rbrace} \bigg \rbrace \; .
\end{equation}
Iterating over this scheme gives after $K$ steps:
\begin{equation}
\phi(X_{k+1}) \leq \phi(X_0) - \sum_{k=0}^K \min \bigg \lbrace \frac{G(X_k)^2}{2 M_{\phi}}, G(X_k) - \frac{M_{\phi}}{2} \cdot \mathbbm{1}_{\lbrace G(X_k) > M_{\phi}	\rbrace} \bigg \rbrace \; .
\end{equation}
Let $\tilde{G}_K = \min_{0\leq k\leq K} G(X_k)$ be the minimum FW gap up to iteration step $K$. Then,
\begin{equation}
\phi(X_{K+1})\leq \phi(X_0) - (K+1) \cdot \min \bigg \lbrace	\frac{\tilde{G}_K ^2}{2 M_{\phi}}, \tilde{G}_K - \frac{M_{\phi}}{2} \cdot \mathbbm{1}_{\lbrace \tilde{G}_K > M_{\phi}	\rbrace} \bigg	\rbrace \; .
\end{equation}
The minimum can attend two different values, i.e. we have to consider two different cases. For the following analysis, we introduce the \emph{initial global suboptimality} 
\begin{equation}
h_0 := \phi(X_0) - \min_{\hat{Z} \in \Xc} \phi(X_k) \; .
\end{equation}
Note that $h_0 \geq \phi(X_0) - \phi(X_{K+1})$.
\begin{enumerate}
\item \underline{$\tilde{G}_K \leq M_{\phi}$:}\\
\begin{align}
\phi(X_{K+1}) &\leq \phi(X_0) - (K+1) \frac{\tilde{G}_K ^2}{2 M_{\phi}} \\
\underbrace{\phi(X_{0}) - \phi(X_{K+1})}_{\leq h_0} &\geq (K+1) \frac{\tilde{G}_K ^2}{2 M_{\phi}} \\
\sqrt{\frac{2 h_0 M_{\phi}}{K+1}} &\geq \tilde{G}_K \; .
\label{eq:FWE-b1}
\end{align}
\item \underline{$\tilde{G}_K > M_{\phi}$:}
\end{enumerate}
\begin{align}
\phi(X_{K+1}) &\leq \phi(X_0) - (K+1) \left( \tilde{G}_K - \frac{M_{\phi}}{2} \right)\\
h_0 &\geq (K+1) \left( \tilde{G}_K - \frac{M_{\phi}}{2} \right)\\
\frac{h_0}{K+1} + \frac{M_{\phi}}{2} &\geq \tilde{G}_K \; . 
\label{eq:FWE-b2}
\end{align}
The bound on the second case can be refined with the observation that $\tilde{G}_K > M_{\phi}$ iff $h_0 > \frac{M_{\phi}}{2}$: Assuming $\tilde{G}_K > M_{\phi}$, we have
\begin{align}
 \frac{h_0}{K+1} + \frac{M_{\phi}}{2} \geq \tilde{G}_K > M_{\phi} \\
 \Rightarrow \frac{h_0}{K+1} > \frac{M_{\phi}}{2} \\
 \Rightarrow \frac{2 h_0}{M_{\phi}} > K+1 \; .
 \label{eq:b2-tight}
 \end{align} 
 However, then $h_0 \leq \frac{M_{\phi}}{2}$ would imply $1>K+1$ contradicting $K>0$; i.e. $\tilde{G}_K > M_{\phi}$ iff $h_0 > \frac{M_{\phi}}{2}$. With this we can rewrite the bound for the second case:
 \begin{align}
 \frac{h_0}{K+1} + \frac{M_{\phi}}{2} &= \frac{h_0}{\sqrt{K+1}} \left(	\frac{1}{\sqrt{K+1}} + \frac{M_{\phi}}{2h_0} \sqrt{K+1}		\right) \\
&\leq^{\dagger} \frac{h_0}{\sqrt{K+1}} \left(	\frac{1}{\sqrt{K+1}} + \sqrt{\frac{M_{\phi}}{2h_0}}		\right) \\
&\leq^{\ddagger} \frac{h_0}{\sqrt{K+1}} \underbrace{\left(	\frac{1}{\sqrt{K+1}} + 1 \right)}_{\leq 2} \\
&\leq \frac{2h_0}{\sqrt{K+1}} \; .
 \end{align}
 Here, ($\dagger$) holds due to Eq. (\ref{eq:b2-tight}) and ($\ddagger$) follows from $h_0>\frac{M_{\phi}}{2}$. \\
 
In summary, we now have
\begin{equation}
\tilde{G}_K \leq \begin{cases}
\frac{2h_0}{\sqrt{K+1}}, &\mbox{if }h_0>\frac{M_{\phi}}{2} \\
\sqrt{\frac{2h_0 M_{\phi}}{K+1}}, &\mbox{else} \; .
\end{cases}
\label{eq:FWE-bounds}
\end{equation}
The upper bounds on the FW gap given by Eq. (\ref{eq:FWE-bounds}) imply convergence rates of $O(1/\sqrt{K})$, completing the proof.
\end{proof}

\subsection{Approximately solving the Riemannian ``linear'' oracle}
We want to prove the following lemma, stated in the main text:
\begin{lemma}\label{lem:smooth-delta}
For a steps size $\eta \in (0,1)$ and accuracy $\delta$, we have
\begin{align*}
\phi(X_{k+1}) \leq \phi(X_k) - \eta \ip{\grad \; \phi(X_k)}{\Exp_{X_k}(z')} + \frac{1}{2}\eta^2 M_{\phi}(1+\delta) \; .
\end{align*}
\end{lemma}
\begin{proof}
We use again the notation $Y = \gamma_{XY}(\eta)$ for a point on the geodesic joining $X$ and $Z'$.  By Lemma~\ref{lem:lip-bound-gen}, we have
\begin{align*}
\phi(Y) \leq \phi(X) + \eta \ip{\grad \; \phi(X)}{\Exp_{X}^{-1}(Z')} + \frac{1}{2} \eta^2 M_{\phi}.
\end{align*}
By construction, we have
\begin{align*}
\ip{\grad \; \phi(X)}{\Exp_{X}^{-1}(Z')} &\leq \min_{Y \in \Xc} \ip{\grad \; \phi(X)}{\Exp_{X}^{-1}(Y)} + \frac{1}{2} \delta \eta M_{\phi} \\
&\leq -(\phi(X) - \phi(X^*)) + \frac{1}{2} \delta \eta M_{\phi} \; .
\end{align*}
Inserting this above,  the claim follows as
\begin{align*}
\phi(Y) &\leq \phi(X) - \eta (\phi(X) - \phi(X^*)) + \frac{1}{2} \delta \eta^2 M_{\phi} + \frac{1}{2} \eta^2 M_{\phi} \\
&\leq \phi(X) - \eta (\phi(X) - \phi(X^*)) + \frac{1}{2} \delta \eta^2 M_{\phi} (1 + \delta) \; .
\end{align*}
\end{proof}

\color{black}

%\section{Mathematical details}
\section{Non-convex Euclidean Frank-Wolfe}
\label{sec:e-lo}
We make a short digression here to mention non-convex Euclidean Frank-Wolfe (\efw) as a potential alternative approach to solving~\eqref{eq:18}. Indeed, the constraint set therein is not only g-convex, it is also convex in the usual sense. Thus, one can also apply an \efw scheme to solve~\eqref{eq:18}, albeit with a slower convergence rate. In general, a g-convex set $\Xc$ need not be Euclidean convex, so this observation does not always apply. 

\efw was recently analyzed in~\citep{LJ16}; the convergence rate reported below adapts one of its main results.
The key difference, however, is that due to g-convexity, we can translate the local result of~\citep{LJ16} into a global one for problem~\eqref{eq:18}.
\begin{theorem}[Convergence \fw-gap (\citep{LJ16})]
  \label{thm:conv-FWE}
  Define $\tilde{G}_k := \min_{0\leq k\leq K} G(X_k)$, where $G(X_k)=\max_{Z_k \in \Xc} \ip{Z_k - X_k}{- \nabla \phi(X_k)}$ is the \fw-gap (i.e., measure of convergence) at $X_k$. Define the curvature constant $$M_{\phi} := \sup_{\substack{X,Y,Z \in \Xc \\ Y=X+s(Z-X)}} \tfrac{2}{s^2} \left[ \phi(Y) - \phi(X) - \ip{\nabla \phi(X)}{Y-X} \right].$$
  Then, after $K$ iterations, \efw satisfies $\tilde{G}_K \leq \frac{\max \; \lbrace 2 h_0, M_\phi \rbrace}{\sqrt{K+1}}$. 
\end{theorem}
The proof (a simple adaption of~\citep{LJ16}) is similar to that of Theorem~\ref{thm:conv-FWR-nonconvex}; therefore, we omit it here. Finally, to implement \efw, we need to also efficiently implement its linear oracle.  Theorem~\ref{thm:trace} below shows how to; the proof is similar to  that of Theorem~\ref{thm:logtrace}.%, so we relegate it to the appendix. 
It is important to note that this linear oracle involves solving a simple SDP, but it is unreasonable to require use of an SDP solver at each iteration.  Theorem~\ref{thm:trace} thus proves to be crucial, because it yields an easily computed closed-form solution. % (for the proof, see appendix).
\begin{theorem}
  \label{thm:trace}
  Let $L, U \in \pd_d$ such that $L \prec U$ and $S \in \H_d$ is arbitrary. Let $U-L=P^*P$, and $PSP^*=Q\Lambda Q^*$. Then, the solution to 
  \begin{equation}
    \label{eq.12}
    \min_{L \preceq Z \preceq U}\quad \trace(SZ),
  \end{equation}
  is given by $Z = L + P^*Q[-\sgn(\Lambda)]_+Q^*P$.
\end{theorem}
\begin{proof}
  First, shift the constraint to $0 \preceq X-L \preceq U-L$; then factorize $U-L=P^*P$, and introduce a new variable $Y=X-L$. Therewith, problem~\eqref{eq.12} becomes
  \begin{equation*}
    \min_{0 \preceq Y \preceq U-L}\quad\trace(S(Y+L)) \; .
  \end{equation*}
  If $L=U$, then clearly $P=0$, and $X=L$ is the solution. Assume thus, $L \prec U$, so that $P$ is invertible. Thus, the above problem further simplifies to
  \begin{equation*}
    \min_{0 \preceq (P^{*})^{-1}YP^{-1} \preceq I}\quad\trace(SY) \; .
  \end{equation*}
  Introduce another variable $Z=P^*YP$ and use circularity of trace to now write  
  \begin{equation*}
    \min_{0 \preceq Z \preceq I}\quad\trace(PSP^*Z) \; .
  \end{equation*}
  To obtain the optimal $Z$, first write the eigenvalue decomposition
  \begin{equation}
    \label{eq.14}
    PSP^* = Q\Lambda Q^*.
  \end{equation}
  Lemma~\ref{lem.weyl} implies that the trace will be maximized when the eigenvectors of $PSP^*$ and $Z$ align and their eigenvalues match up. Since $0 \preceq Z \preceq I$, we therefore see that $Z=QDQ^*$ is an optimal solution, where $D$ is diagonal with entries
  \begin{equation}
    \label{eq.15}
    d_{ii} =
    \begin{cases}
      1 &\text{if}\ \lambda_i(Y) < 0\\
      0 &\text{if}\ \lambda_i(Y) \geq 0
    \end{cases}\quad\implies D=[-\sgn(\Lambda)]_+.
  \end{equation}
  Undoing the variable substitutions we obtain $X=L+P^*Q[-\sgn(\Lambda)]_+Q^*P$ as desired.
\end{proof}
\begin{rmk}\normalfont
  Computing the optimal $X$ requires 1 Cholesky factorization, 5 matrix multiplications and 1 eigenvector decomposition. The theoretical complexity of the Euclidean Linear Oracle can therefore be estimated as $O(N^3)$. On our machine, eigenvector  decomposition is approximately 8--12 times slower than matrix multiplication. So the total flop count is approximately $\approx \frac{1}{3} N^3 + 5\times 2N^3 + 20N^3 \approx 33N^3$.
\end{rmk}
An implementation of \efw for the computation of Riemannian centroids is shown in Algorithm~\ref{alg.efw}. Experimental results for \efw in comparison with \rfw and state-of-the-art Riemannian optimization methods can be found in the main text (see section~\ref{sec:exp-karcher}).
\begin{algorithm}[H]
  \caption{\efw for fast Geometric mean}
  \label{alg.efw}
  \begin{algorithmic}[1]
     \State $(A_1,\ldots,A_N)$, $\bm{w} \in \rplus^N$
     \State $\bar{X} \approx \argmin_{X > 0} \nlsum_iw_i \riem^2(X,A_i)$
     \State $\beta = \min_{1 \leq i \leq N} \lambda_{{\rm min}} (A_i)$
     \For {$k=0,1,\dots$}
     	\State Compute gradient: $\; \nabla\phi(X_k) = X_k^{-1}\bigl(\nlsum_iw_i \log(X_kA_i^{-1})\bigr)$
        \State Compute $Z_k$: $\; Z_k \gets \argmin_{H \preceq Z \preceq A}\ip{\nabla\phi(X_k)}{Z-X_k}$
        \State Let $\alpha_k \gets \frac{2}{k+2}$.
        \State Update $X$: $\; X_{k+1} \gets X_k + \alpha_k(Z_k - X_k)$.
     \EndFor
     \State \textbf{return} $\bar{X}=X_k$
   \end{algorithmic}
 \end{algorithm}

\section{Generating matrices}

\subsection{Positive Definite Matrices}
For testing our methods in the well-conditioned regime, we generate matrices $\lbrace A_i \rbrace_{i=1}^n \in \mathbb{P}_d$ by sampling real matrices of dimension $d$ uniformly at random $A_i \sim \mathcal{U}(\mathbb{R}^{d \times d})$
and multiplying each with its transpose
\begin{align*}
A_i &\gets A_i A_i^T \; .
\end{align*}
This gives well-conditioned, positive definite matrices.

Furthermore, we sample $m$ matrices $U_i \sim \mathcal{U}(\mathbb{R}^{d \times d})$ with a rank deficit, i.e. ${\rm rank}(U)<d$. Then, setting
\begin{align*}
B_i &\gets \delta I + U_i U_i^T \; 
\end{align*}
with $\delta$ being small yields ill-conditioned matrices.

\subsection{Special Orthogonal group}
First, we use a publicly available implementation by Ofek Shilon\footnote{https://www.mathworks.com/matlabcentral/fileexchange/11783-randorthmat} to sample uniformly from the manifold of orthogonal matrices, i.e., $\lbrace A_i \rbrace_{i=1}^n \in O(d)$. The method generates uniformly distributed orthogonal matrices by applying the Gram-Schmidt process to random matrices. All matrices in our sample will have determinant $\det(A_i)=\pm 1$. To map our sample to the special orthogonal group (with determinant 1), we multiply the first row of all matrices with negative determinant by $-1$: Let $a_{1*}^i$ be the first row of $A_i$. If $\det(A_i)=-1$, set
\begin{align*}
a_{1*}^i \gets -a_{1*}^i \; .
\end{align*}

\section{Additional Experiments}
\subsection{Comparison against classic gradient-based methods}
In addition to the comparison with state-of-the-art methods in the main text, we compared against two classic gradient-based methods:
\begin{enumerate}
\item \emph{Steepest Decent (SD)} is an iterative, first order minimization algorithm with line search that has risen to great popularity in machine learning. In each iteration, the method evaluates the gradient of the function and steps towards a local minimum proportional to the negative of the gradient. We use an implementation provided by the toolbox \emph{Manopt}~\citep{manopt}.
\item \emph{Conjugate Gradient (CG)} is an iterative optimization algorithm with plane search, i.e. instead of evaluating the gradient one evaluates a linear combination of the gradient and the descent step of the previous iteration. We again use an implementation provided by  \emph{Manopt}.
\end{enumerate}
\begin{figure}[!htbp]
\centering
    \includegraphics[scale=0.13]{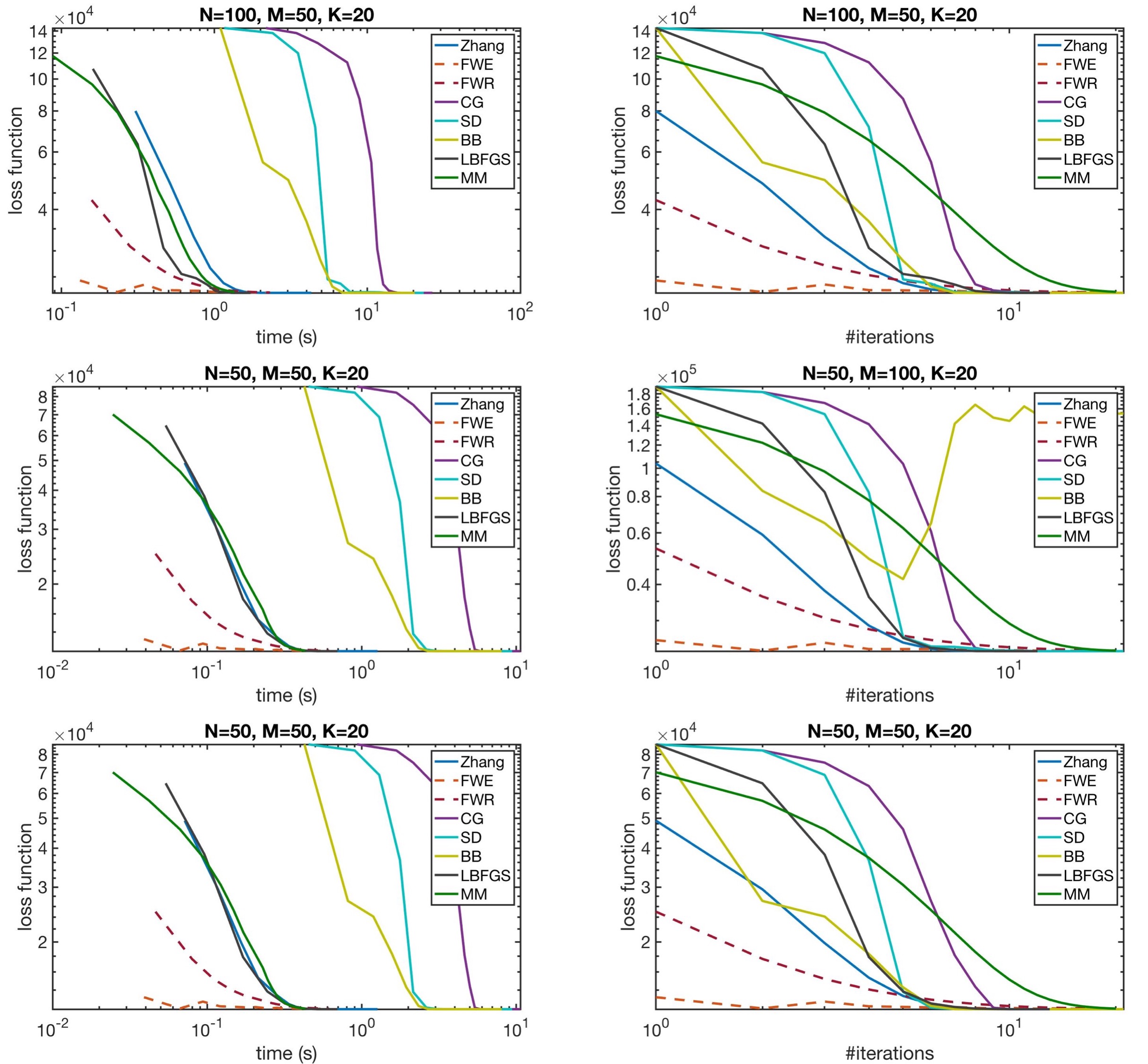} 
     \caption{Performance of \efw and \rfw in comparison with state-of-the-art methods for well-conditioned inputs of different sizes ($N$: size of matrices, $M$: number of matrices, $K$: maximum number of iterations). All tasks are initialized with the harmonic mean $x_0=HM$.}
     \label{fig:all}
\end{figure}
In Fig.~\ref{fig:all}, we report the loss both with respect to CPU time and with respect to the number of iterations. This allows for a more objective comparison with implementations from the toolbox Manopt (see discussion in the main text).

\subsection{Additional comparison with state-of-the-art methods}
For a more objective comparison, we include here also the loss with respect to the number of iterations for (i) different parameter choices (Fig.~\ref{fig:param-it}) and (ii) different initializations (Fig.~\ref{fig:init-it}).
\begin{figure}[!htbp]
\centering
    \includegraphics[scale=0.13]{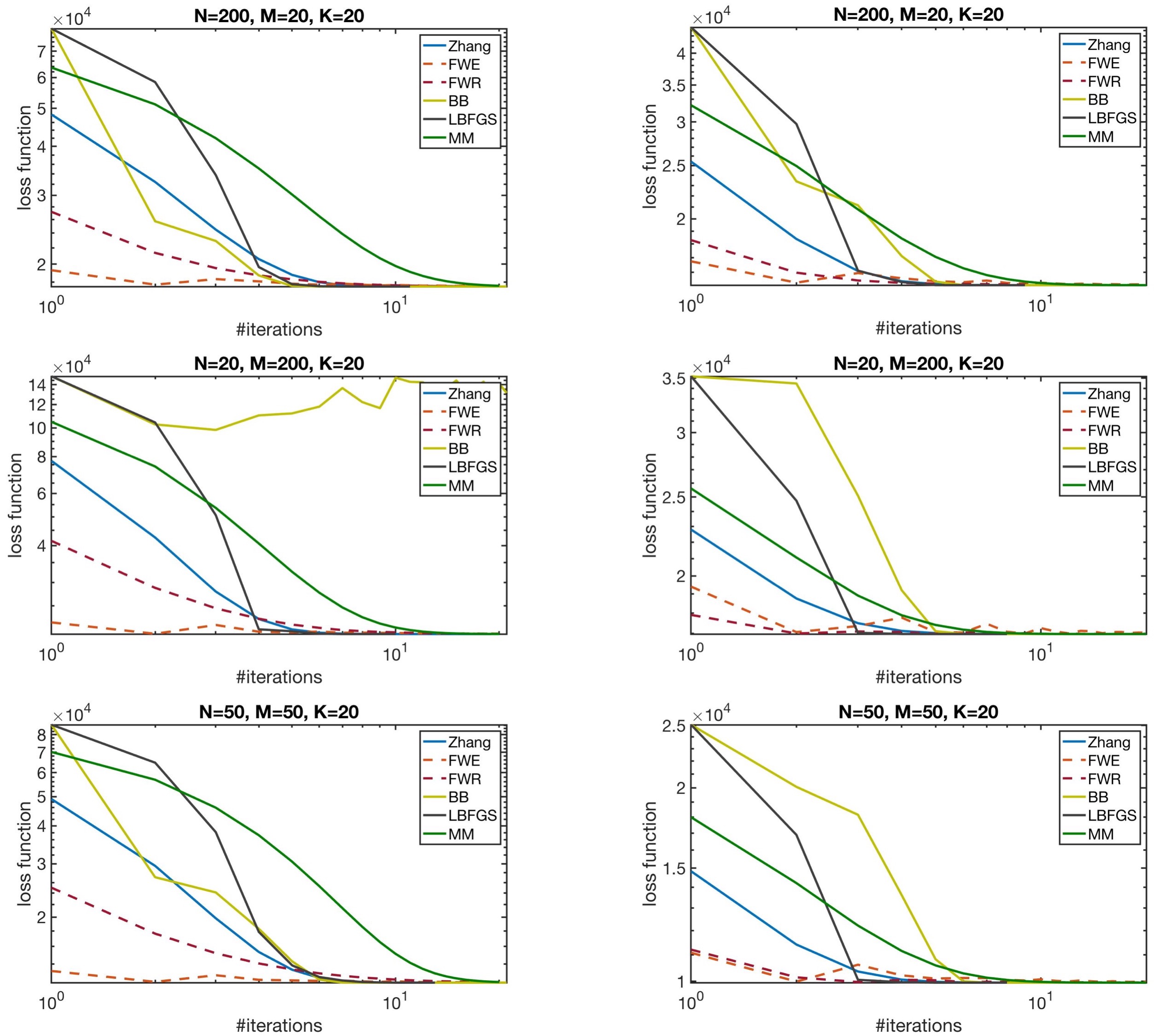} 
     \caption{Performance of \efw and \rfw in comparison with state-of-the-art methods for well-conditioned (left) and ill-conditioned (right) inputs of different sizes ($N$: size of matrices, $M$: number of matrices, $K$: maximum number of iterations). All tasks are initialized with the harmonic mean $x_0=HM$.}
     \label{fig:param-it}
\end{figure}
\begin{figure}[!htbp]
\centering
    \includegraphics[scale=0.15]{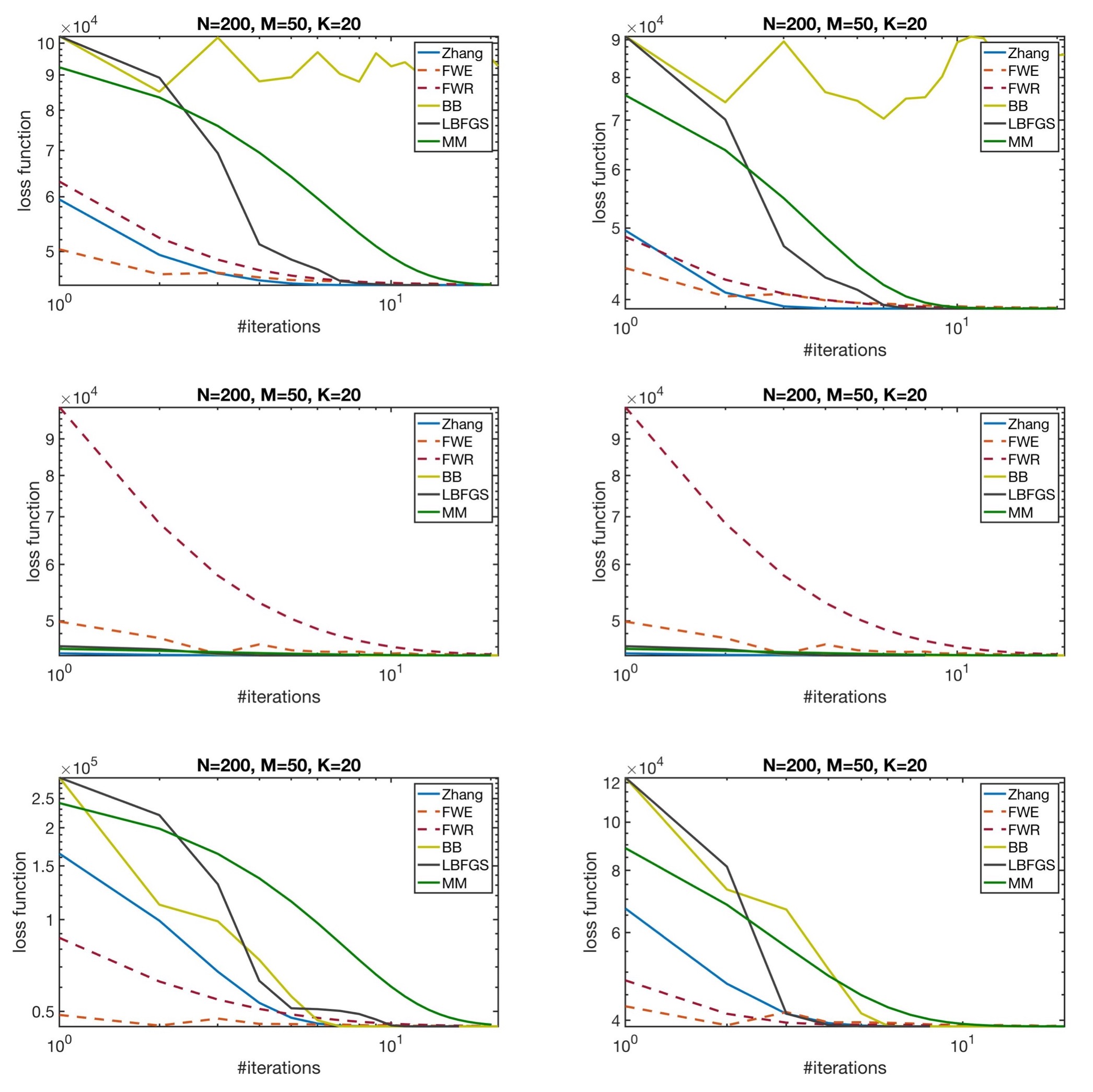} 
     \caption{Performance of \efw and \rfw in comparison with state-of-the-art methods for well-conditioned (left) and ill-conditioned (right) inputs with different initializations: $x_0=A_1$ for $A_1 \in A$ (left), $x_0=\frac{1}{2}(AM + HM)$ (middle) and $x_0=HM$ (right).}
     \label{fig:init-it}
\end{figure}
\begin{figure}
\begin{center}
\begin{tabular}[b]{lrr}\hline
     \textbf{Method} & \textbf{\# calls to grad} & \textbf{\# calls to cost} \\ \hline
      \efw & 30 & 0 \\
      \rfw & 30 & 0 \\
      RBB &  17 &  35 \\
      R-LBFGS & 30 & 49 \\
      MM & 30 & 0 \\
      Zhang & 30 &  0 \\
      \hline
    \end{tabular}
    \caption{Number of calls to the gradient and cost functions until reaching convergence, averaged over ten experiments with $N=40$, $M=10$ and $K=30$. Note, that \efw/ \rfw outperform most other methods in this example which is partially due to avoiding internal calls to the cost function, significantly increasing the efficiency of both methods.}
    \label{tab:calls}
    \end{center}
\end{figure}

\end{document}